\documentclass[5p]{elsarticle}

\usepackage[intlimits]{amsmath}
\usepackage{amsthm,amssymb,amsfonts}
\usepackage[english]{babel}

\usepackage{graphics} 
\usepackage{epsfig} 
\theoremstyle{plain}
\newtheorem{assumption}{Assumption}

\newtheorem{lemma}{Lemma}
\newtheorem{theorem}{Theorem}

\theoremstyle{remark}
\newtheorem{remark}{Remark}
\biboptions{sort&compress}
\DeclareMathOperator*{\esssup}{ess\,sup}

\journal{Automatica}

\begin{document}
	
\begin{frontmatter}
		
\title{Backstepping Control of Continua of Linear Hyperbolic PDEs and
  Application to Stabilization of Large-Scale $n+m$ Coupled Hyperbolic PDE
  Systems\tnoteref{t1}} 
\tnotetext[t1]{Funded by the European Union (ERC, C-NORA, 101088147). Views and 
opinions expressed are however those of the authors only and do not necessarily reflect those of 
the European Union or the European Research Council Executive Agency. Neither the European 
Union nor the granting authority can be held responsible for them.}
		
\author[1]{Jukka-Pekka Humaloja}
\author[1]{Nikolaos Bekiaris-Liberis}
		
\affiliation[1]{organization={Department of Electrical and Computer Engineering, Technical 
University of Crete},
	addressline={University Campus, Akrotiri}, 
	city={Chania},
	postcode={73100}, 
	country={Greece}}
		
\begin{abstract}
 We develop a backstepping control design for a class of continuum systems of linear hyperbolic 
 PDEs, described by a coupled system of an ensemble of rightward transporting PDEs and a 
 (finite) system of $m$ leftward transporting PDEs. The key analysis challenge of the design is to 
 establish well-posedness of the resulting ensemble of kernel equations, since they evolve on a 
 prismatic (3-D) domain and inherit the potential discontinuities of the kernels for the case of 
 $n+m$ hyperbolic systems. We resolve this challenge generalizing the well-posedness analysis 
 of Hu, Di Meglio, Vazquez, and Krstic to continua of general, heterodirectional hyperbolic 
 PDE  systems, while also constructing a proper Lyapunov functional.

Since the motivation for addressing such PDE systems continua comes from the objective to 
develop computationally tractable control designs for large-scale PDE systems, we then introduce 
a methodology for stabilization of general $n+m$ hyperbolic systems, constructing stabilizing 
backstepping control kernels based on the continuum kernels derived from the continuum system 
counterpart. This control design procedure is enabled by establishing that, as $n$ grows, the 
continuum backstepping control kernels can approximate (in certain sense) the exact kernels, and 
thus, they remain stabilizing (as formally proven). This approach guarantees that complexity of 
computation of stabilizing kernels does not grow with the number $n$ of PDE systems 
components. We further establish that the solutions to the $n+m$ PDE system converge, as 
$n\to\infty$, to the solutions of the corresponding continuum PDE system.

We also provide a numerical example in which the continuum kernels can be obtained in closed 
form (in contrast to the large-scale kernels), thus resulting in minimum complexity of control 
kernels computation, which illustrates the potential computational benefits of our approach.

\end{abstract}
		
		
\end{frontmatter}

\section{Introduction}

\subsection{Motivation}

Stabilization of large-scale systems of general, $n+m$ heterodirectional linear hyperbolic PDEs 
can be achieved via backstepping, see, for example, \cite{AnfAamBook, AurBre22, AurDiM16, 
CorHuL17, DiMArg18, HuLDiM16, HuLVaz19, RamZwaIFAC13}. Such large-scale systems of 
hyperbolic PDEs may be 
utilized to describe the dynamics of various systems with practical importance. In particular, they 
can be utilized to describe, traffic flow dynamics in large traffic networks \cite{FriGot22, GotHer21,
TumCan22, ZhaLua22}, as well as in multi-lane \cite{HerKla03, YuHKrs21} or 
multi-class traffic \cite{BurYuH21, MohRam17}; blood flow dynamics in cardiovascular networks 
consisting of 
interconnected arterial segments \cite{BikPhd, ReyMer09}; epidemics spreading dynamics in 
various 
geographical regions and among different age groups \cite{BasCorBook, GuaPri20, IanBook,  
KitBes22}; dynamics of 
multi-phase flows in oil drilling applications \cite{DiMKaaACC11}; and water networks dynamics 
\cite{BasCorBook, DiaDia17}. Complexity 
of computation of stabilizing backstepping kernels may, in general, grow with the number of PDE 
systems components \cite{HumBek24arxivb, HumBekCDC24}, which may, in fact, be 
alleviated constructing backstepping feedback laws based on continua PDE systems counterparts 
 \cite{HumBek24arxivb, HumBekCDC24}. Consequently, motivated by this and the 
practical significance of considering large-scale systems of hyperbolic PDEs, we address the 
problem of design of computationally tractable backstepping feedback laws for large-scale 
systems of $n+m$ heterodirectional linear hyperbolic PDEs, via introduction of a control design 
procedure that relies on development of backstepping control laws for continua PDE systems 
counterparts. 

\subsection{Literature}

The first result on backstepping stabilization of a class of continua of hyperbolic PDE systems was 
developed in \cite{AllKrs25}, while a formal connection between the class of systems considered 
in \cite{AllKrs25} and the class of $n+1$ linear hyperbolic systems \cite{DiMVaz13} (for large 
$n$), 
as well as the 
application of the control design originally developed for the continuum system to the large-scale 
counterpart, were made in \cite{HumBek24arxiv, HumBekCDC24}. Therefore, besides 
\cite{AllKrs25} and \cite{HumBek24arxiv, HumBekCDC24}, the present paper is related to the 
results on backstepping stabilization of $n+m$ linear hyperbolic systems, see, for example, 
\cite{AnfAamBook, AurDiM16, CorHuL17, DiMArg18, HuLDiM16, HuLVaz19}, as well as to results in 
which PDE ensembles may arise as result of employment of Fourier transform, see, for example, 
\cite{VazKrsBook} (that deals with parabolic PDEs). In addition, as the actual motivation for our 
developments is to address computational complexity of backstepping designs for large-scale 
hyperbolic systems, papers related to computation of backstepping kernels are also relevant, in 
particular, \cite{BhaShi24} that introduces a neural operators-based computation method, 
\cite{AurMor19} that 
presents a late-lumping-based approach, and \cite{VazCheCDC23} that relies on power series 
representations 
of the kernels (even though these results do not explicitly aim at addressing computational 
complexity with respect to increasing number of systems components). Here we address the 
previously unattempted problems of backstepping control design for the continuum counterpart 
of a large-scale system of $n+m$ hyperbolic PDEs and its application for the stabilization of the 
original large-scale system.

\subsection{Contributions}

 We start considering a continuum PDE system that corresponds to the $n+m$ hyperbolic system 
 as $n\to\infty$\footnote{The problem of stabilization of the system as $m\to\infty$ is a different 
 problem that requires a quite different treatment; see our discussion in Section~\ref{sec:conc} 
 (second paragraph).} 
 for which we employ the continuum PDE backstepping method. This gives rise to a 
 continuum plus $m$ kernel equations that are defined on a prismatic (3-D) domain that arises by 
 continuating the triangular (2-D) domain of definition of the respective $n+m$ kernel equations. 
 We establish well-posedness of the kernel equations treating them on each 3-D subdomain that 
 is spanned along the direction of the ensemble variable from subdomains of the 2-D triangular 
 space on which the kernels do not feature discontinuities. This allows us to then show continuity 
 of the respective characteristic projections and to employ the successive approximations 
 approach on each 3-D subdomain, thus generalizing the well-posedness results from 
 \cite{HuLDiM16, HuLVaz19} and \cite{AllKrs25} for the $n+m$ and $\infty+1$ cases, respectively, 
 to the case of a continuum 
 plus $m$ ($\infty+m$) kernels. Such a generalization is highly nontrivial and requires a delicate 
 technical treatment as it inherits the technical intricacies of both going from $n+1$ to $n+m$ 
 systems, in particular, the fact that the kernels may feature discontinuities, and going from a 
 system with finite components to a continuum, in particular, having to deal with PDEs defined on 
 2-D domains instead of vector-valued 1-D PDEs. We establish exponential stability (in $L^2$) of 
 the closed-loop system, constructing a Lyapunov functional. 

We then consider the large-scale $n+m$ system counterpart for which we design a feedback law 
employing the continuum kernels (evaluated at $n$ points for each control input). We establish 
that the closed-loop system is exponentially stable (in $L^2$) by showing that, for sufficiently 
large $n$, the exact backstepping kernels can be approximated to any desired accuracy by the 
continuum kernels. The proof relies on construction of a sequence of backstepping kernels that is 
defined such that each kernel in the sequence matches the exact kernel (in a piecewise manner 
with respect to the ensemble variable), while then showing that this sequence converges to the 
continuum kernel. This in turn implies that the approximation error of the exact control kernels can 
be made arbitrarily small for sufficiently large $n$. This gives rise to a closed-loop system that is 
affected by a bounded vanishing perturbation with a bound that can be made arbitrarily small, and 
thus, the closed-loop system remains exponentially stable, which we show constructing a 
Lyapunov functional. We further provide a formal convergence result establishing the exact 
convergence properties of the actual solutions of the $n+m$ PDE system to the solutions of the 
continuum counterpart. This is achieved constructing a sequence of solutions that matches the 
solutions to the $n+m$ system in a piecewise manner with respect to the ensemble variable and 
then showing that this sequence converges to the solutions of the continuum system. We provide 
a numerical example of an $n+m$ system for which the exact control kernels do not exhibit a 
closed-form solution, but the continuum kernels of the respective $\infty+m$ system do, 
illustrating the computational complexity benefits of employing the continuum kernels for 
stabilization of the large-scale system.

\subsection{Organization}

The paper is organized as follows. In Section~\ref{sec:infmstab}, we present a backstepping 
control law for $\infty+m$ hyperbolic systems and show exponential stability of the closed-loop 
system. In 
Section~\ref{sec:wpk}, we show that the backstepping kernels employed in the control law are 
well-posed. In Section~\ref{sec:nmstab}, we show that the $\infty+m$ kernels can be 
used in constructing exponentially stabilizing control laws for large-scale $n+m$ hyperbolic 
systems. In Section~\ref{sec:nmapp}, we formally show the convergence of the solution of the 
$n+m$ PDE system to the solution of the
the $\infty+m$ PDE system. In Section~\ref{sec:ex}, we validate 
the theoretical developments on a numerical  example, illustrating in simulation that stabilization  
of the $n+m$ system can be achieved employing the continuum kernels-based controller. In 
Section~\ref{sec:conc}, we provide concluding remarks and discuss open problems.

\subsection{Notation}

We use the standard notation $L^2(\Omega; \mathbb{R})$ for real-valued
Lebesgue integrable functions on a domain $\Omega \subset \mathbb{R}^d$ for some $d \geq 1$. 
For conciseness, we occasionally use shorthand $L^2$ for $L^2([0,1];\mathbb{R})$. The notations
$L^\infty(\Omega;\mathbb{R}),C(\Omega; \mathbb{R})$, and $C^1(\Omega; \mathbb{R})$ denote 
essentially bounded, continuous, and
continuously differentiable functions, respectively, on $\Omega$. Moreover, the notation
$f \in L_{\rm{loc}}^2([0,+\infty);\mathbb{R})$ means that $f \in L^2([0,
a];\mathbb{R})$ for any $a\in \mathbb{N}$. We 
denote vectors and matrices by bold symbols, and $\|\cdot\|_\infty,\|\cdot\|_1$ denote the 
maximum 
absolute row and column sums, respectively, of a matrix (or a vector). For any $n,m \in 
\mathbb{N}$, we denote by $E$ the 
space $L^2([0,1]; \mathbb{R}^{n+m})$ equipped with the inner product 
\begin{align}
\label{eq:eip}
\left\langle \left( 
\begin{smallmatrix}
\mathbf{u}_1 \\ \mathbf{v}_1
\end{smallmatrix}
\right), \left( 
\begin{smallmatrix}
\mathbf{u}_2\\ \mathbf{v}_2
\end{smallmatrix}
\right) \right\rangle_E
& = \nonumber \\ \int\limits_0^1 \left(\frac{1}{n}
\sum_{i=1}^nu_1^i(x)u_2^i(x) +
  \sum_{j=1}^mv_1^j(x)v_2^j(x)\right)dx,
&
\end{align}
which induces the norm $\left\|
\cdot
\right\|_E = \sqrt{\left<\cdot,\cdot\right>_E}$. We also define the continuum version of $E$ as 
$n\to\infty$ by
$E_c = L^2([0,1];L^2([0,1];\mathbb{R}))\times L^2([0,1];\mathbb{R}^m)$, (i.e., $\mathbb{R}^n$ 
becomes $L^2([0,1];\mathbb{R})$ as $n\to\infty$) equipped with the inner 
product
\begin{align}
\label{eq:ecip}
\left\langle \left( 
\begin{smallmatrix}
u_1 \\ \mathbf{v}_1
\end{smallmatrix}
\right), \left( 
\begin{smallmatrix}
u_2\\ \mathbf{v}_2
\end{smallmatrix}
\right) \right\rangle_{E_c}
& = \nonumber \\
 \int\limits_0^1\left(\int\limits_0^1 u_1(x,y)u_2(x,y)dy + \sum_{j=1}^mv_1^j(x)v_2^j(x)\right)dx, &
\end{align}
which coincides with $L^2([0,1]^2;\mathbb{R})\times L^2([0,1];\mathbb{R}^m)$.
Moreover, we say that a 
system is exponentially stable on $E$ (resp. on $E_c$) if, for any initial condition $z_0\in E$ (resp. 
$z_0\in E_c$), the (weak) solution 
$z \in C([0,\infty); E)$ (resp. $z \in C([0,\infty); E_c)$) of the system satisfies $\|z(t)\|_E \leq 
Me^{-ct}\|z_0\|_E$ (resp. $\|z(t)\|_{E_c} \leq 
Me^{-ct}\|z_0\|_{E_c}$) for some constants $M,c > 0$ that are independent of $z_0$. Finally, we 
denote by 
$\mathcal{T}$ and $\mathcal{P}$ the triangular and prismatic, respectively, sets
\begin{subequations}
\begin{align}
\mathcal{T} & = \left\{ (x,\xi) \in [0,1]^2: 0 \leq \xi \leq x \leq 1 \right\}, \\
\mathcal{P} & = \left\{(x,\xi,y) \in [0,1]^3: (x,\xi) \in \mathcal{T} \right\}.
\end{align}
\end{subequations}

\section{Stabilization of Continua $\infty + m$ Systems} \label{sec:infmstab}

\subsection{Continua $\infty+m$ Systems of Hyperbolic PDEs}

The considered class of systems can be thought of as the continuum counterpart of $n+m$ 
hyperbolic systems in the limit case $n\to\infty$.\footnote{This aspect is 
considered formally in Section~\ref{sec:nmapp}.} However, instead of considering a countably 
infinite number as $n\to\infty$, we replace the $n$-part by an (uncountably infinite) ensemble 
over the variable $y \in 
[0,1]$. Thus, the considered class of continuum systems is of the form
\begin{subequations}
\label{eq:infm}%
\begin{align}
  u_t(t,x,y) + \lambda(x,y)u_x(t,x,y)   & = \nonumber \\
 \int\limits_0^1\sigma(x,y,\eta)u(t,x,\eta)d\eta +
   \mathbf{W}(x,y)\mathbf{v}(t,x), & \label{eq:infm1} \\
     \mathbf{v}_t(t,x) - \mathbf{M}(x)\mathbf{v}_x(t,x)
   & =  \nonumber \\
   \int\limits_0^1\pmb{\Theta}(x,y)u(t,x,y)dy + \pmb{\Psi}(x)\mathbf{v}(t,x),
   & \label{eq:infm2}
\end{align}
\end{subequations}
with boundary conditions 
\begin{subequations}
\label{eq:infmbc}%
\begin{align}
u(t,0,y) & = \mathbf{Q}(y)\mathbf{v}(t,0), \\
\mathbf{v}(t,1) & = \mathbf{U}(t),
\end{align}
\end{subequations}
for almost every $y \in [0,1]$. Here we employ the matrix
notation for $\mathbf{v}, \mathbf{U}, \mathbf{M}, \pmb{\Theta},
\pmb{\Psi}, \mathbf{W}$, and $\mathbf{Q}$ for the sake of 
conciseness, that is, $\mathbf{v} = \left( v^j \right)_{j=1}^m$,
$\mathbf{U} = \left( U^j \right)_{j=1}^m$, and the parameters are as follows.
\begin{assumption}
\label{ass:infm}
The parameters of \eqref{eq:infm}, \eqref{eq:infmbc} are such that
\begin{subequations}
\label{eq:infmparam}
\begin{align}
\mathbf{M} & = \operatorname{diag}(\mu_j)_{j=1}^m
\in C^1([0,1]; \mathbb{R}^{m\times m}), \\
\pmb{\Theta} & = \left( \theta_j \right)_{j=1}^m \in
C([0,1];L^2([0,1];\mathbb{R}^m)), \\
\pmb{\Psi} & = \left( \psi_{i,j} \right)_{i,j=1}^m \in C([0,1];
\mathbb{R}^{m\times m}), \\
\mathbf{W} & = 
\begin{bmatrix}
	W_1 & \cdots  & W_m
\end{bmatrix} \in C([0,1];L^2([0,1];\mathbb{R}^{1\times m})), \\
\mathbf{Q} & = 
\begin{bmatrix}
	Q_1 & \cdots & Q_{m}
\end{bmatrix} \in L^2([0,1]; \mathbb{R}^{1\times m}),
\end{align}
\end{subequations}
with $\lambda  \in C^1([0,1]^2;\mathbb{R})$ and  $\sigma \in
C([0, 1]; L^2([0,1]^2;\mathbb{R}))$. Moreover, $\mu_m(x) > 0$ and $\lambda(x,y) > 0$ uniformly 
for all $x,y \in [0,1]$, and additionally
\begin{equation}
	\label{eq:muass}
	\min_{x\in [0,1]}\mu_j(x) > \max_{x\in[0,1]}\mu_{j+1}(x),
\end{equation}
for all $j = 1,\ldots, m-1$. Finally, $\psi_{j,j} = 0$ for all $j = 
1,\ldots,m$.\footnote{This comes without loss of generality, as such terms
	can be removed using a change of variables (see also, e.g., \cite{HuLDiM16, HuLVaz19}).}
\end{assumption}
\begin{remark}
\label{rem:infmwp}
Under Assumption~\ref{ass:infm}, it can be shown by using the same arguments as in \cite[Prop. 
B.1]{HumBek24arxiv} that the system \eqref{eq:infm}, \eqref{eq:infmbc} is well-posed on $E_c$. 
That is, for any initial conditions $u_0 \in L^2([0,1];L^2([0,1];\mathbb{R})), \mathbf{v}_0 \in 
L^2([0,1];\mathbb{R}^m)$ and input $\mathbf{U} \in L_{\rm loc}^2([0, +\infty); \mathbb{R}^m)$, 
there is a unique (weak) solution to \eqref{eq:infm}, \eqref{eq:infmbc} satisfying $(u, \mathbf{v}) 
\in C([0,+\infty); E_c)$.
\end{remark}

\subsection{Continuum Backstepping Kernel Equations}

The target system
for the continuum Volterra backstepping transformation is chosen~as
\begin{subequations}
\label{eq:infmts}%
\begin{align}
  \alpha_t(t,x,y) + \lambda(x,y)\alpha_x(t,x,y)
  & = \nonumber \\
   \int\limits_0^1\sigma(x,y,\eta)\alpha(t,x,\eta)d\eta +
  \mathbf{W}(x,y)\pmb{\beta}(t,x) & \nonumber \\
   + \int\limits_0^1\int\limits_0^x C^+(x,\xi,y,\eta)\alpha(t,\xi,\eta)d\xi d\eta & \nonumber \\
   +  \int\limits_0^x \mathbf{C}^-(x,\xi,y)\pmb{\beta}(t,\xi)d\xi, \label{eq:infmts1} \\
    \pmb{\beta}_t(t,x) - \mathbf{M}(x)\pmb{\beta}_x(t,x)
  & =  \mathbf{G}(x)\pmb{\beta}(t,0), \label{eq:infmts2}
\end{align}
\end{subequations}
with boundary conditions
\begin{subequations}
\label{eq:infmtsbc}%
\begin{align}
\alpha(t,0,y) & = \mathbf{Q}(y)\pmb{\beta}(t,0), \label{eq:infmtsbc1} \\
\pmb{\beta}(t,1) & = \pmb{0},  \label{eq:infmtsbc2}
\end{align}
\end{subequations}
for (almost) all $y\in[0,1]$, where $C^+ \in L^\infty(\mathcal{T}; 
L^2([0,1]^2;\mathbb{R}))$, $\mathbf{C}^- \in L^\infty(\mathcal{T}; L^2([0,1]; \mathbb{R}^{1\times 
m}))$, and 
$\mathbf{G} \in L^\infty([0,1]; \mathbb{R}^{m\times m})$ is of the form
\begin{equation}
	\label{eq:GG}
	\mathbf{G}(x) = 
	\begin{bmatrix}
		0 & \cdots & \cdots & 0 \\
		G_{2,1}(x) & \ddots  & \ddots & \vdots \\
		\vdots & \ddots & \ddots & \vdots  \\
		G_{m,1}(x)  & \cdots  & G_{m,m-1}(x) & 0
	\end{bmatrix}.
\end{equation}
In order to map \eqref{eq:infm},
\eqref{eq:infmbc} into \eqref{eq:infmts},  
\eqref{eq:infmtsbc}, we employ the following continuum Volterra transformation 
\begin{subequations}
  \label{eq:infmV}%
\begin{align}
	\alpha(t,x,y) & = u(t,x,y)  \label{eq:infmV1} \\
  \pmb{\beta}(t,x) & = \mathbf{v}(t,x) -
  \int\limits_0^x\mathbf{L}(x,\xi)\mathbf{v}(t,\xi)d\xi \nonumber
  \\
  & \qquad - \int\limits_0^x \int\limits_0^1
    \mathbf{K}(x,\xi,y)u(t,\xi,y)dyd\xi, \label{eq:infmV2}
\end{align}
\end{subequations}
where $\mathbf{L}\in L^\infty(\mathcal{T}; \mathbb{R}^{m\times m})$ and $\mathbf{K} 
\in L^\infty(\mathcal{T}; L^2([0,1]; \mathbb{R}^m))$ are the backstepping 
kernels.\footnote{We, in fact, show in Section~\ref{sec:wpk} that $\mathbf{K}$ and $\mathbf{L}$ 
are piecewise continuous in $(x,\xi) \in \mathcal{T}$, so that evaluation along the boundaries of 
$\mathcal{T}$ is well-defined.}

Derivation of the continuum kernels equations is provided in \ref{app:infmk}. We obtain that
$\mathbf{L}$ and $\mathbf{K}$ need to satisfy the following kernel equations
\begin{subequations}
\label{eq:infmk}%
\begin{align}
	  \mathbf{M}(x)\mathbf{K}_x(x,\xi,y) -
	\mathbf{K}_{\xi}(x,\xi,y)\lambda(\xi,y) -
	\mathbf{K}(x,\xi,y)\lambda_{\xi}(\xi,y) & = \nonumber \\
	\mathbf{L}(x,\xi)\pmb{\Theta}(\xi,y) +
	\int\limits_0^1\mathbf{K}(x,\xi,\eta)\sigma(\xi,\eta,y)d\eta,
	& \label{eq:infmkb} \\
  \mathbf{M}(x)\mathbf{L}_x(x,\xi) +
  \mathbf{L}_{\xi}(x,\xi)\mathbf{M}(\xi) +
  \mathbf{L}(x,\xi)\mathbf{M}'(\xi)
  & = \nonumber \\
  \mathbf{L}(x,\xi)\pmb{\Psi}(\xi) + \int\limits_0^1
  \mathbf{K}(x,\xi,y)\mathbf{W}(\xi,y)dy, & \label{eq:infmka}
\end{align}
\end{subequations}
with boundary conditions
\begin{subequations}
\label{eq:infmkbc}%
\begin{align}
  \mathbf{M}(x)\mathbf{L}(x,x) - \mathbf{L}(x,x)\mathbf{M}(x) +
  \pmb{\Psi}(x)
  & = 0,\label{eq:infmkbca} \\
 \mathbf{K}(x,x,y)\lambda(x,y) + \mathbf{M}(x)\mathbf{K}(x,x,y)
+ \pmb{\Theta}(x,y) & = 0,
\label{eq:infmkbcb} \\
\mathbf{L}(x,0)\mathbf{M}(0) -
\int\limits_0^1\mathbf{K}(x,0,y)\lambda(0,y)\mathbf{Q}(y)dy & = \mathbf{G}(x),
\label{eq:infmkbce}
&
\end{align}
\end{subequations}
for almost all $0 \leq \xi\leq x \leq 1$ and $y,\eta \in
[0,1]$. More precisely, \eqref{eq:infmkbce} splits into two parts
\begin{subequations}
\label{eq:infmkbcetmp}
\begin{align}
	 \forall i\leq j: & & L_{i,j}(x,0) & =  \frac{1}{\mu_j(0)} \int\limits_0^1
	K_i(x,0,y)\lambda(0,y)Q_j(y)dy, \label{eq:infmkbcetmp1} \\
	\forall i > j: & & G_{i,j}(x) & = \frac{1}{\mu_j(0)} \int\limits_0^1
	K_i(x,0,y)\lambda(0,y)Q_j(y)dy, \label{eq:infmkbcetmp2}
\end{align}
\end{subequations}
where \eqref{eq:infmkbcetmp1} acts as a boundary condition for \eqref{eq:infmk} and 
\eqref{eq:infmkbcetmp2} defines the nonzero elements of $\mathbf{G}$ shown in \eqref{eq:GG}. 
Similarly to 
\cite{HuLDiM16, HuLVaz19}, we also impose additional, artificial boundary conditions, to ensure
the well-posedness of the kernel equations, as follows
\begin{equation}
\label{eq:infmkabc}%
\forall j < i: \qquad L_{i,j}(1,\xi) = l_{i,j}^{(1)}(\xi),
\end{equation}
where the functions $l^{(1)}_{i,j}$ are chosen such
that a $C^0$ compatibility condition\footnote{While $C^2$ compatibility
	conditions are sought in \cite{HuLVaz19} for obtaining (piecewise) $C^2$ kernels,
	for our purposes $C^0$ compatibility conditions are enough.} is satisfied on
$(x,\xi)=(1,1)$.\footnote{For the 
$\mathbf{L}$ kernels, a compatibility condition cannot (generally) be
satisfied on $(x,\xi) = (0,0)$ due to \eqref{eq:infmkbca} and
\eqref{eq:infmkbce}, \eqref{eq:infmkbcb}.} Thus,
consistently with \eqref{eq:infmkbca}, we impose
\begin{equation}
\label{eq:infmabc0}%
  l^{(1)}_{i,j}(1) =-\frac{\psi_{i,j}(1)}{\mu_i(1) - \mu_j(1)},
\end{equation}
for all $j < i$. The well-posedness
of the kernel equations  
\eqref{eq:infmk}--\eqref{eq:infmabc0} is considered in Section~\ref{sec:wpk}.

\subsection{Backstepping Feedback Law and Stability Result}

The backstepping control law for $j=1,\ldots,m$ is given~by
\begin{align}
	\label{eq:infmU}
	U^j(t) & =\int\limits_0^1\int\limits_0^1
	K_j(1,\xi,y)u(t,\xi,y)dyd\xi \nonumber \\
	& \qquad +\int\limits_0^1 \sum_{i=1}^mL_{j,i}(1,\xi)v^i(t,\xi) d\xi,
\end{align}
which stabilizes \eqref{eq:infm}, \eqref{eq:infmbc} by Theorem~\ref{thm:stab}.
\begin{theorem}
	\label{thm:stab}
Under Assumption~\ref{ass:infm}, the control law \eqref{eq:infmU}
exponentially stabilizes the system \eqref{eq:infm}, \eqref{eq:infmbc} on $E_c$.
\begin{proof}
{\em Well-Posedness:} We first establish that the target system \eqref{eq:infmts}, 
\eqref{eq:infmtsbc}  has a well-posed solution on $E_c$, which we achieve by utilizing feedback 
results for the well-posed system \eqref{eq:infm}, \eqref{eq:infmbc} (see 
Remark~\ref{rem:infmwp}). By \cite[Sect. 
10.1]{TucWeiBook}, we can express the well-posed,
boundary-controlled PDE \eqref{eq:infm}, \eqref{eq:infmbc} as a
well-posed abstract Cauchy problem $\dot{z}(t) = Az(t)+BU(t)$ on the
Hilbert space $E_c$, where $z = (u, \mathbf{v})$, the system
$\dot{z}(t)=Az(t)$ corresponds to \eqref{eq:infm} with the homogeneous
boundary condition from \eqref{eq:infmbc} through the domain of $A$,
and $BU(t)$ corresponds to the boundary control in
\eqref{eq:infmbc}. Now, expressing the backstepping control law
\eqref{eq:infmU} as $U(t) = Fz(t)$, the closed-loop dynamics of \eqref{eq:infm}, \eqref{eq:infmbc} 
under the control law \eqref{eq:infmU} are given by
$\dot{z}(t) = (A + BF)z(t)$. As $F$ is a bounded linear operator from
$E_c$ to $\mathbb{R}^m$, the operator $A+BF$ is the generator of a
strongly continuous semigroup by \cite[Cor. 5.5.1]{TucWeiBook}, and
hence, the dynamics $\dot{z}(t) = (A+BF)z(t)$ have a well-posed solution on $E_c$. Now, by 
applying the linear, bounded state transformation \eqref{eq:infmV} (see 
Theorem~\ref{thm:infmkwp} in Section~\ref{sec:wpk}), we have that the target 
system \eqref{eq:infmts}, \eqref{eq:infmtsbc} has a well-posed solution on $E_c$ as well. 

{\em Lyapunov Stability:} Now we show that the (weak; see \cite[Rem. 4.6]{HumBek24arxiv} for 
details on the fact that existence and uniqueness of a weak solution suffices for making our 
Lya- punov-based arguments legitimate)
solution to \eqref{eq:infmts}, 
\eqref{eq:infmtsbc} decays exponentially to zero, which by the invertibility of the transform 
\eqref{eq:infmV} (see Lemma~\ref{lem:invk} in \ref{app:invk}) implies that the system 
\eqref{eq:infm}, \eqref{eq:infmbc} under the control law \eqref{eq:infmU} is exponentially stable. 
Inspired by \cite[Prop. 2.1]{HuLVaz19}, the candidate Lyapunov functional with parameters
$\delta, \mathbf{D} = \operatorname{diag}(D_1,\ldots,D_m) > 0$ is taken as 
\begin{align}
  \label{eq:lyap}
  V(t) & = \int\limits_0^1 \int\limits_0^1 e^{-\delta
         x}\frac{\alpha^2(t,x,y)}{\lambda(x,y)}dydx \nonumber \\
& \quad  + \int\limits_0^1
  e^{\delta x}\pmb{\beta}^T(t,x)\mathbf{D}\mathbf{M}^{-1}(x)\pmb{\beta}(t,x)dx.
\end{align}
Computing $\dot{V}(t)$ and integrating by parts in $x$ gives 
\begin{align}
  \label{eq:lyapd}
  \dot{V}(t)
  & = \left[-e^{-\delta x}\|\alpha(t,x,\cdot)\|^2_{L^2} +
    e^{\delta x}\|\pmb{\beta}(t,x)\|^2_{\mathbf{D}}
    \right]_0^1 \nonumber \\
  & \quad - \delta\int\limits_0^1 \left(e^{-\delta
    x}\|a(t,x,\cdot)\|_{L^2}^2 + e^{\delta
    x}\|\pmb{\beta}(t,x)\|_{\mathbf{D}}^2 \right)dx \nonumber \\
  & \quad \resizebox{.86\columnwidth}{!}{$ \displaystyle
  	+ 2\int\limits_0^1 \int\limits_0^1 \int\limits_0^1e^{-\delta
    x}\frac{\alpha(t,x,y)}{\lambda(x,y)}\sigma(x,\eta,y)\alpha(t,x,\eta)
    d\eta dy dx$} \nonumber \\
  & \quad + 2\int\limits_0^1 \int\limits_0^1 e^{-\delta
    x}\frac{\alpha(t,x,y)}{\lambda(x,y)}
    \mathbf{W}(x,y)\pmb{\beta}(t,x)dydx \nonumber \\
  & \quad \resizebox{.7\columnwidth}{!}{$ \displaystyle
    + 2\int\limits_0^1 \int\limits_0^1 \int\limits_0^1
    \int\limits_0^x e^{-\delta x}
    \frac{\alpha(t,x,y)}{\lambda(x,y)}C^+(x,\xi,y,\eta)\alpha(t,\xi,\eta)$}
    d\xi d\eta dy dx \nonumber \\
  & \quad + 2\int\limits_0^1 \int\limits_0^1 \int\limits_0^x
  e^{-\delta x}\frac{\alpha(t,x,y)}{\lambda(x,y)}
  \mathbf{C}^-(x,\xi,y)\pmb{\beta}(t,\xi)d\xi dy dx \nonumber \\
  & \quad 
    + \int\limits_0^1  e^{\delta x}\pmb{\beta}^T(t,x)\left(
    \mathbf{D}\mathbf{M}^{-1}(x)\mathbf{G}(x) \right. \nonumber \\
 &  \left. \qquad \qquad +\mathbf{G}^T(x)\mathbf{M}^{-1}(x)\mathbf{D}
    \right)\pmb{\beta}(t,0)dx ,
\end{align}
where $\|\cdot\|_{\mathbf{D}}^2 = \left\langle \cdot, \mathcal{D}\cdot
\right\rangle_{R^m}$ denotes the $\mathbf{D}$-weighted inner product. Using the
following bounds (that exist by Assumption~\ref{ass:infm} using Theorem~\ref{thm:infmkwp} and 
Lemma~\ref{lem:C-} in \ref{app:infmk})
\begin{subequations}
\label{eq:lyapb}%
\begin{align}
  m_{\lambda} & = \min_{x,y \in [0,1]} \lambda(x,y), \\
  m_{\mu}  &= \min_{j \in \left\{ 1,\ldots,m \right\}}\min_{x\in
            [0,1]} \mu_j(x), \\
M_{\sigma} & = \max_{x\in [0,1]} \left\|
              \int\limits_0^1\sigma(x,\cdot,\eta)d\eta
              \right\|_{L^2}, \label{eq:Ms} \\
M_W & = \max_{j= \left\{ 1,\ldots,m
      \right\}}\max_{x\in[0,1]}\|W_j(x,\cdot)\|_{L^2}, \label{eq:MW} \\
M_{C^+} & = \esssup_{(x,\xi) \in \mathcal{T}} \left\|
           \int\limits_0^1C^+(x,\xi,\cdot,\eta)d\eta \right\|_{L^2},
            \\
  M_{C^-} & = \max_{j\in \left\{ 1,\ldots,m \right\}}\esssup_{(x,\xi)
            \in \mathcal{T}} \|C_j^-(x,\xi,\cdot)\|_{L^2}, \\
  M_G & = \max_{i,j\in \left\{ 1,\ldots,m \right\}}\esssup_{x\in
        [0,1]}\left| G_{ij}(x) \right|, \\
  M_Q & = \max_{j=1,\ldots,m}\|Q_j\|_{L^2},
\end{align}
\end{subequations}
the boundary conditions \eqref{eq:infmtsbc}, the Cauchy-Schwartz
inequality, and $2\left\langle f,g \right\rangle_{L^2} \leq
\|f\|_{L^2}^2 +\|g\|^2_{L^2}$ for any $f,g \in L^2$, we can estimate \eqref{eq:lyapd}
as 
\begin{align}
  \label{eq:lyapdest}
  \dot{V}(t)
  & \leq - \pmb{\beta}^T(t,0) \left( \mathbf{D} -
   M_Q^2I_{m\times m}
    \right)\pmb{\beta}(t,0)
    \nonumber \\
  & \quad - \delta\int\limits_0^1 \left(e^{-\delta
    x}\|a(t,x,\cdot)\|_{L^2}^2 + e^{\delta
    x}\|\pmb{\beta}(t,x)\|_{\mathbf{D}}^2 \right)dx \nonumber \\
  & \quad + 2\int\limits_0^1e^{-\delta x}
    \frac{M_{\sigma}+M_{C^+}}{m_{\lambda}}\|\alpha(t,x,\cdot)\|_{L^2}^2dx
    \nonumber \\
  & \quad + \int\limits_0^1e^{-\delta x} \left(
    \frac{\|\alpha(t,x,\cdot)\|_{L^2}^2}{m_{\lambda}^2} +
    M_W^2\| \pmb{\beta}(t,x)\|_{\mathbb{R}^m}^2 \right)dx \nonumber \\
  & \quad + \int\limits_0^1e^{-\delta x} \left(
    \frac{\|\alpha(t,x,\cdot)\|_{L^2}^2}{m_{\lambda}^2} + M_{C^-}^2\|
    \pmb{\beta}(t,x)\|_{\mathbb{R}^m}^2  \right)dx \nonumber \\
  & \quad + mM_G\int\limits_0^1e^{\delta x}
    \pmb{\beta}^T(t,x)\mathbf{D}\mathbf{M}^{-1}(x)\pmb{\beta}(t,x)dx
    \nonumber \\
&   \quad + \frac{mM_Ge^{\delta}}{\delta m_{\mu}}
                   \pmb{\beta}^T(t,0) \mathbf{F} \pmb{\beta}(t,0),
\end{align}
where $\mathbf{F} = \operatorname{diag}\left( F_1, \ldots, F_m
\right)$ with 
\begin{equation}
  \label{eq:Fmat}
  F_j = 
\begin{cases}
  \sum_{i=j+1}^m D_i, &  1 \leq j \leq m-1, \\
  0, & j = m,
\end{cases}
\end{equation}
where we employ the lower-triangular structure of $\mathbf{G}$ given in \eqref{eq:GG}  on the
last two lines of \eqref{eq:lyapdest}. Now, $\dot{V}(t)$ can be guaranteed to be
negative definite by choosing $\delta$ and $\mathbf{D}$ such that 
\begin{subequations}
\label{eq:lparam}%
\begin{align}
\delta & > \max \resizebox{.8\columnwidth}{!}{$\displaystyle \left\{
         \frac{2m_{\lambda}(M_{\sigma} + M_{C^+}) + 
         2}{m_{\lambda}^2}, \frac{M_W^2 + M_{C^-}^2 + mM_G}{m_{\mu}}\right\}$}, \\
D_j & > \left(1 + (m-j)\frac{mM_Ge^{\delta}}{\delta m_{\mu}}
      \right)\max \left\{ M_Q^2, 1 \right\},
\end{align}
\end{subequations}
for all $j \in \{1,\ldots,m\}$. More specifically, by defining 
\begin{equation}
	c_V = \delta  - \max \resizebox{.75\columnwidth}{!}{$\displaystyle \left\{
		\frac{2m_{\lambda}(M_{\sigma} + M_{C^+}) + 
			2}{m_{\lambda}^2}, \frac{M_W^2 + M_{C^-}^2 + mM_G}{m_{\mu}}\right\}$}, 
\end{equation}
we have 
\begin{equation}
	\dot{V}(t) \leq - \frac{c_V}{\max\left\{M_\mu, M_\lambda\right\}}V(t),
\end{equation}
where
\begin{equation}
	\label{eq:Mlammu}
		M_{\lambda} = \max_{x,y \in [0,1]} \lambda(x,y), \quad
		M_{\mu} = \max_{j=\left\{ 1,\ldots,m
			\right\}} \max_{x \in [0,1]} \mu_j(x),
\end{equation}
which shows that the target system \eqref{eq:infmts},
\eqref{eq:infmtsbc} is exponentially stable. Thus, due to the invertibility of the transform 
\eqref{eq:infmV} established in Lemma~\ref{lem:invk} in \ref{app:invk}, the control law
\eqref{eq:infmU} exponentially stabilizes \eqref{eq:infm}, \eqref{eq:infmbc}.
\end{proof}
\end{theorem}

\section{Well-Posedness of the Continuum Kernels} \label{sec:wpk}

\begin{theorem}
  \label{thm:infmkwp}
  Under Assumption~\ref{ass:infm}, the continuum kernel equations 
  \eqref{eq:infmk}--\eqref{eq:infmabc0} 
  have a well-posed solution $\mathbf{K} \in L^\infty(\mathcal{T};L^2([0,1]; \mathbb{R}^m))$ and 
  $\mathbf{L} \in L^\infty(\mathcal{T}; \mathbb{R}^{m\times m})$. Moreover, the solution is 
  piecewise continuous in $(x,\xi) \in \mathcal{T}$, where the set of discontinuities is of measure 
  zero.
\end{theorem}
  
The proof is presented at the end of this section by utilizing the following lemmas. First, the 
kernels are split into  subdomains to deal with the potential discontinuity in the $L_{i,j}$ kernels 
for $i<j$ stemming from $(x,\xi) = (0,0)$ due to the boundary conditions \eqref{eq:infmkbca} and 
\eqref{eq:infmkbce}, \eqref{eq:infmkbcb}. Once the kernels are split into subdomains, the 
resulting kernel equations can be solved by transforming them into integral 
equations along the characteristic curves and solving these integral equations by using the 
method of successive approximations combining \cite[Sect. VI]{AllKrs25} and \cite[Sect. 
VI]{HuLDiM16}. In particular, we need to ensure continuity of the characteristic curves such that 
the successive approximations for the $K_i$ kernels, for $i=1,\ldots,m$,  are $L^2$ in $y$ 
for almost all $(x,\xi) \in \mathcal{T}$.
	
\begin{lemma}[Splitting the kernels into subdomains of continuity]
\label{lem:split}
The kernel equations \eqref{eq:infmk} can be equivalently written in $L^\infty(\mathcal{T}; E_c)$
 as
\begin{subequations}
	\label{eq:infmkLk}%
	\begin{align}
	\resizebox{.95\columnwidth}{!}{$\mu_i(x)\partial_xK_i^p(x,\xi,y) -
		\partial_{\xi}K_i^p(x,\xi,y)\lambda(\xi,y) -
		K_i^p(x,\xi,y)\lambda_{\xi}(\xi,y)$} & = \nonumber \\
	\sum_{\ell=1}^m L_{i,\ell}^p(x,\xi)\theta_\ell(\xi,y) + \int\limits_0^1
	K^p_i(x,\xi,\eta)\sigma(\xi,\eta,y)d\eta, \label{eq:infmkLk2}
	\end{align}
	\begin{align}
	\mu_i(x)\partial_xL_{i,j}^p(x,\xi) +
	\mu_j(\xi)\partial_{\xi}L_{i,j}^p(x,\xi) + \mu'_j(x)L_{i,j}^p(x,\xi) 
	& = \nonumber \\
	\sum_{\ell=1}^m L_{i,\ell}^p(x,\xi)\psi_{\ell,j}(\xi) +
	\int\limits_0^1K_i^p(x,\xi,y)W_j(\xi,y)dy, \label{eq:infmkLk1}
	\end{align}
\end{subequations}
for $1 \leq i \leq p \leq m$ and $j=1,\ldots,m$,
where $L_{i,j}^p, K_i^p$ denote the restrictions of the
kernels to $\mathcal{T}_i^p$ and $\mathcal{P}_i^p$,
respectively, defined as
\begin{subequations}
	\label{eq:TPs}
	\begin{align}
		\mathcal{T}_i^p & = \left\{ (x,\xi)\in [0,1]^{2}:
		\rho_i^{p+1}(x) \leq \xi \leq \rho_i^p(x) \right\},  \label{eq:Ts} \\
		\mathcal{P}_i^p & = \left\{ (x,\xi,y)\in [0,1]^{3}:
		(x,\xi) \in \mathcal{T}_i^p \right\},  \label{eq:Ps} 
	\end{align}
\end{subequations}
where $\rho_i^{m+1} = 0$ for all $i = 1,\ldots,m$ and 
\begin{equation}
	\label{eq:rho}
	\rho_i^p(x) = \phi_p^{-1}(\phi_i(x)),\footnote{These are the characteristic curves of 
	\eqref{eq:infmkLk1}, which are strictly increasing in $x$ and satisfy $0 = \rho_i^{m+1}(x) < 
	\rho_i^m(x) < \cdots < \rho_i^i(x) = x$ for all $1 \leq i \leq p \leq m$ and $x \in [0,1]$ by 
	\eqref{eq:muass}.}
\end{equation}
for $1 \leq i \leq p \leq m$ with
\begin{equation}
	\label{eq:phim}
	\phi_i(x) = \int\limits_0^x \frac{ds}{\mu_i(s)}, \qquad i = 1,\ldots,m.
\end{equation}
The boundary conditions for \eqref{eq:infmkLk} are given by
\begin{subequations}
	\label{eq:infmkLkbc}%
	\begin{align}
		\forall j \neq i: & & L_{i,j}^i(x,x)
		& = -\frac{\psi_{i,j}(x)}{\mu_i(x) - \mu_j(x)}, \label{eq:infmkLkbca}
		\\
		\forall i: & & K_i^i(x,x,y)
		& = -\frac{\theta_i(x,y)}{\lambda(x,y) +
			\mu_i(x)}, \label{eq:infmkLkbcb} \\
		\forall i\leq j: & & L_{i,j}^m(x,0) & =
		\resizebox{.58\columnwidth}{!}{$\displaystyle \frac{1}{\mu_j(0)} \int\limits_0^1
			K_i^m(x,0,y)\lambda(0,y)Q_j(y)dy$}, \label{eq:infmkLkbcc}
	\end{align}
\end{subequations}
for $i,j = 1,\ldots,m$, with the artificial boundary conditions 
\begin{equation}
	\label{eq:abct}
	L_{ij}^p(1,\xi) = l_{i,j}^{(1)}(\xi),
\end{equation}
for all $\xi \in [\rho_i^{p+1}(1), \rho_i^p(1)]$, $p = i,\ldots,m$, and $1\leq j < i \leq m$. Moreover,  
the segmented kernels $K_i^p, L_{i,j}^p$ are subject to continuity conditions
\begin{subequations}
	\label{eq:infmkbcc}
	\begin{align}
		\forall i < p,\forall j\neq p: & &  L_{i,j}^{p-1}(x,\rho_i^p(x))
		& = L_{i,j}^p(x,\rho_i^p(x)), \label{eq:infmkLkcbca} \\
		\forall i < p: & & K_i^{p-1}(x,\rho_i^p(x),y)
		& = K_i^p(x,\rho_i^p(x),y), \label{eq:infmkLkcbcb}
	\end{align} 
\end{subequations}
for all $i,j = 1,\ldots,m$, $i < p \leq m$, and $x, y \in [0,1]$.
\begin{proof}
After splitting the kernels into the $\mathcal{T}_i^p$ and $\mathcal{P}_i^p$ segments, the 
transformation \eqref{eq:infmV2} can be rewritten
componentwise for $i=1,\ldots,m$ as 
\begin{align}
	\label{eq:betai}
	\beta_i(t,x) & = v_i(t,x) - \sum_{j=1}^m\sum_{p=i}^m
	\int\limits_{\rho_i^{p+1}(x)}^{\rho_i^p(x)}L_{i,j}^p(x,\xi)v_j(t,\xi)d\xi
	\nonumber \\
	& \qquad - \sum_{p=i}^m \int\limits_{\rho_i^{p+1}(x)}^{\rho_i^p(x)}\int\limits_0^1 
	K_i^p(x,\xi,y)u(t,\xi,y)dyd\xi.
\end{align}
The kernel equations \eqref{eq:infmkLk} are obtained by  inserting
\eqref{eq:betai} to \eqref{eq:infmts2} and integrating by parts once
(similarly to \ref{app:infmk}). In fact, the kernel equations 
\eqref{eq:infmkLk} are exactly of the same form as \eqref{eq:infmk} (written componentwise), and 
the boundary conditions \eqref{eq:infmkLkbc}, \eqref{eq:abct} correspond to \eqref{eq:infmkbca}, 
\eqref{eq:infmkbcb},  \eqref{eq:infmkbcetmp1}, and \eqref{eq:infmkabc} along the respective 
boundaries (see Figure~\ref{fig:Tseg} for an illustration of the $\mathcal{T}_i^p$ segments). 
Thus, the only difference to \eqref{eq:infmk}--\eqref{eq:infmkabc} are the continuity conditions 
\eqref{eq:infmkbcc}, which arise due to the segmentation of $\mathcal{T}$ when differentiating 
\eqref{eq:betai} in $x$ and integrating by parts once.
\end{proof}
\end{lemma}

\begin{figure}[htbp]
\centerline{\includegraphics[scale=.7]{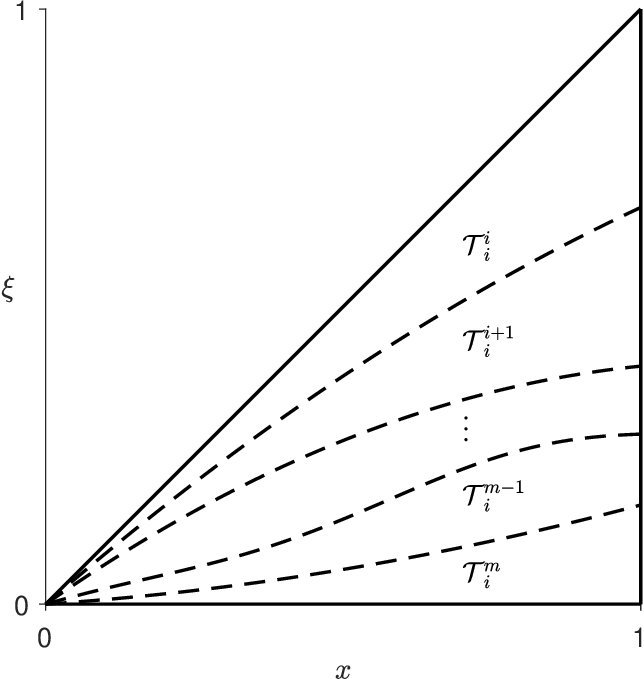}}
\caption{Illustration of the segments $\mathcal{T}_i^p$ for $1\leq i \leq p
	\leq m$. The dashed lines are the characteristic curves $\xi = \rho_i^p(x)$ for $i < p \leq m$.}
\label{fig:Tseg}
\end{figure}

The kernel equations \eqref{eq:infmkLk} for $L_{i,j}^p$ and $K_i^p$ on the segments 
$\mathcal{T}_i^p$ and $\mathcal{P}_i^p$ with boundary conditions 
\eqref{eq:infmkLkbc}--\eqref{eq:infmkbcc} can be 
transformed into integral equations. In order to do this, in Lemma~\ref{lem:ccc} we solve the 
characteristic projections for \eqref{eq:infmkLk}.

\begin{lemma}[Continuity of characteristic projections]
	\label{lem:ccc}
The characteristic projections for the kernel equations \eqref{eq:infmkLk} are continuous on 
$\mathcal{T}_i^p$ and $\mathcal{P}_i^p$ for all $1 \leq i \leq p \leq m$.
\begin{proof}
As $\lambda$ is assumed to be in $C^1([0,1]^2;\mathbb{R})$, we can argue pointwise in $y\in 
[0,1]$ and solve the characteristic projections for the $K_i^p$ kernels from the following Cauchy 
problems on $s \in [0, s_{i,p}^f(y)]$ for arbitrary, fixed $y\in[0,1]$ and $1 \leq i \leq p \leq m$
\begin{subequations}
\label{eq:infmkkic}%
\begin{align}
  \frac{d}{ds}\hat{x}_{i,p}(s,y) & = -\mu_i(\hat{x}_{i,p}(s,y)), \\
  \frac{d}{ds}\hat{\xi}_{i,p}(s,y) & = \lambda(\hat{\xi}_{i,p}(s,y),y),
\end{align}
\end{subequations}
with boundary conditions $\hat{x}_{i,p}(0,y) = x,
\hat{x}_{i,p}\left(s^f_{i,p},y\right) = \hat{x}^f_{i,p}(y), \hat{\xi}_{i,p}(0,y)
= \xi, \hat{\xi}_{i,p}\left(s^f_{i,p},y\right) = \hat{\xi}^f_{i,p}(y)$. Since $\mu_i$ and 
$\lambda(\cdot,y)$ are 
continuously differentiable and positive by Assumption~\ref{ass:infm}, \eqref{eq:infmkkic} has a 
unique (local in $s$) solution for any $(x,\xi) \in \mathcal{T}_i^p$ (and each $y$) by 
Picard---Lindel\"of theorem \cite[Thm 2.2]{TesBook}, 
where $\hat{x}_{i,p}$ is strictly decreasing in $s$ and $\xi_{i,p}$ is strictly increasing in $s$. 
Thus, for 
any initial condition $(x,\xi) \in \mathcal{T}_i^p$, the solution to \eqref{eq:infmkkic} tends towards 
the boundary $\xi =\rho_i^p(x)$, where it terminates at $s = s_{i,p}^f(y)$ with the terminal 
condition $\left(\hat{x}_{i,p}^f,\hat{\xi}_{i,p}^f\right)$, and  the corresponding boundary condition 
is given by
\eqref{eq:infmkLkbcb} for $i = p$,  or by \eqref{eq:infmkLkcbcb} for $i < p$.
Considering that we have only split the domain of the kernel equations in the
$(x,\xi)$ plane, we can employ the same continuity arguments, not only
in $y$ but also in $x$ and $\xi$, as in \cite[Lem. 4]{AllKrs25} on each $\mathcal{P}_i^p$
for all $1 \leq i \leq p \leq m$. Thus, the characteristic projections solving
\eqref{eq:infmkkic} are continuous on each $\mathcal{P}_i^p$,
particularly as $\lambda$ and $\mu_i$ are continuously differentiable by 
Assumption~\ref{ass:infm}.

The characteristic projections for the $L_{i,j}^p$ kernels are
analogous to the $\ell_{i,j}^p$ kernels encountered in the $n+m$ case. Thus, this observation 
allows us to study continuity of the characteristic projections for the $L_{i,j}^p$ kernels similarly 
to \cite[Thm A.1]{HuLVaz19} and
\cite[Sect. VI.A.2]{HuLDiM16}. To elaborate,
for all $i,j = 1,\ldots,m$ and $p = i,\ldots,m$, the characteristic
projections for the $L_{i,j}^p$ kernels are solutions of the following
Cauchy problems on $s \in [0, s_{i,j,p}^f]$ 
\begin{subequations}
\label{eq:infmkLic}%
\begin{align}
  \frac{d}{ds}\hat{x}_{i,j,p}(s)
  & = \epsilon_{i,j}\mu_i(\hat{x}_{i,j,p}(s)), \\
   \frac{d}{ds}\hat{\xi}_{i,j,p}(s)
  & = \epsilon_{i,j}\mu_j(\hat{\xi}_{i,j,p}(s)), 
\end{align}
\end{subequations}
with boundary conditions $\hat{x}_{i,j,p}(0) = x,
\hat{x}_{i,j,p}\left(s_{i,j,p}^f\right) = \hat{x}_{i,j,p}^f$, $\hat{\xi}_{i,j,p}(0)
= \xi, \hat{\xi}_{i,j,p}\left(s_{i,j,p}^f\right) =
\hat{\xi}_{i,j,p,}^f$, and $\epsilon_{i,j}$ defined~as 
\begin{equation}
  \label{eq:epsij}
  \epsilon_{i,j} = 
\begin{cases}
  1, & i > j \\
  -1, & i \leq j
\end{cases}.
\end{equation}
For initial condition $(x,\xi) \in
\mathcal{T}_i^p$, the location of the terminal condition
$\left(\hat{x}_{i,j,p}^f, \hat{\xi}_{i,j,p}^f\right)$ depends on $i,j$ and
$p$ (cf. \cite[Figs. 4--6]{HuLDiM16}) as follows.
\begin{itemize}
\item For $i>j$, the terminal condition
is located either on $x=1$ with boundary condition \eqref{eq:abct} for
$i \leq p \leq m$, or on $\xi = x$ with boundary condition
\eqref{eq:infmkLkbca} for $p = i$, or on $\xi = \rho_i^p(x)$ with
boundary condition \eqref{eq:infmkLkcbca} for $i < p \leq m$.
\item For $i=j$, the terminal condition
is located on $\xi = 0$ for $p=m$ with boundary condition
\eqref{eq:infmkLkbcc}, and on $\xi = \rho_i^{p+1}(x)$ for $i \leq p <
m$ with boundary condition \eqref{eq:infmkLkcbca} (for $p \to p+1$).
\item For $i < j$, the terminal
condition is located on $\xi = x$ for $p = i$ with boundary condition
\eqref{eq:infmkLkbca}, on $\xi = \rho_i^p(x)$ for $i < p < j$ with
boundary condition \eqref{eq:infmkLkcbca}, on $\xi = 0$ for $p =
m$ with boundary condition \eqref{eq:infmkLkbcc}, and on $\xi =
\rho_i^{p+1}(x)$ for $j \leq p < m$ with boundary condition
\eqref{eq:infmkLkcbca} (for $p \to p+1$).
\end{itemize}
Thus, there exist unique, continuous characteristic projections as the solutions to 
\eqref{eq:infmkLic} on $s \in [0, s_{i,j,p}^f]$, as every $\mu_i$ is continuously 
differentiable by Assumption~\ref{ass:infm}.
\end{proof}
\end{lemma}

As the final step, we transform the kernel equations \eqref{eq:infmkLk} into integral equations 
along the characteristic curves. By virtue of Lemma~\ref{lem:ccc}, we can then proceed with the 
method of successive approximations to obtain the unique
continuous kernels $K_i^p, L_{i,j}^p$ solving
\eqref{eq:infmkLk}--\eqref{eq:infmkbcc} on each $\mathcal{T}_i ^p$ by Lemma~\ref{lem:sa}. 
Towards this end, integrating \eqref{eq:infmkLk} along the characteristic curves and
plugging in the boundary conditions
\eqref{eq:infmkLkbc}--\eqref{eq:infmkbcc} gives
\begin{subequations}
\label{eq:infmKie}%
\begin{align}
  K_i^p\left(x,\xi,y\right) -
  B^1_{i,p}\left(x^f_{i,p}\left(y\right),y\right) & = \nonumber \\
  -\int\limits_0^{s_{i,p}^f(y)}\left(
  K_i^p\left(\hat{x}_{i,p}\left(s,y\right),\hat{\xi}_{i,p}\left(s,y\right),y\right)
  \lambda_{\xi}\left(\hat{\xi}_{i,p}\left(s,y\right),y\right) \right.
  & \nonumber \\
+
\int\limits_0^1K_i^p\left(\hat{x}_{i,p}\left(s,y\right),\hat{\xi}_{i,p}\left(s,y\right),\eta\right)
  \sigma\left(\hat{\xi}_{i,p}\left(s,y\right),\eta,y\right)d\eta &
  \nonumber \\ \left.
   +
  \sum_{\ell=1}^mL_{i,\ell}^p\left(\hat{x}_{i,j,p}\left(s\right),\hat{\xi}_{i,j,p}\left(s\right)\right)
  \theta_{\ell}\left(\hat{\xi}_{i,j,p}\left(s\right),y\right)
  \right)ds,
& \label{eq:infmKie1} \\
  L_{i,j}^p\left(x,\xi\right) -
  B^2_{i,j,p}\left(\hat{\star}_{i,j,p}\left(s^f_{i,j,p}\right)\right)
  & = \nonumber \\ +\epsilon_{i,j}\int\limits_0^{s_{i,j,p}^f} \left(
  \mu_j'\left(\hat{x}_{i,j,p}\left(s\right)\right)
  L_{i,j}^p\left(\hat{x}_{i,j,p}\left(s\right),\hat{\xi}_{i,j,p}\left(s\right)\right)
  \right.  & \nonumber \\
  -  \int\limits_0^1K_i^p\left(\hat{x}_{i,j,p}\left(s\right),\hat{\xi}_{i,j,p}\left(s\right),y\right)
  W_j\left(\hat{\xi}_{i,j,p}\left(s\right),y\right)dy & \nonumber \\
  \left.
  -\sum_{\ell=1}^mL_{i,\ell}^p\left(\hat{x}_{i,j,p}\left(s\right),\hat{\xi}_{i,j,p}(s)\right)
  \psi_{\ell,j}\left(\hat{\xi}_{i,j,p}(s)\right)
  \right)ds, \label{eq:infmKie2}
\end{align}
\end{subequations} where, for $j = 1,\ldots,m$ and $1 \leq i \leq p \leq m$,
\begin{subequations}
\label{eq:bctmp}%
\begin{align} B_{i,p}^1(x,y) & = \begin{cases}
  -\frac{\theta_i(x,y)}{\lambda(x,y) + \mu_i(x)}, & p = i \\
  K_i^{p-1}(x,\rho_i^p(x),y), & p > i
\end{cases}, \label{eq:bctmp1}
\end{align}
\begin{align}
B_{i,j,p}^2(\star) & =
\begin{cases}
  -\frac{\psi_{i,j}(x)}{\mu_i(x) - \mu_j(x)}, & p= i, i \neq j \\
  l_{i,j}^{(1)}(\xi), & p \geq i > j \\
  L_{i,j}^{p-1}(x,\rho_i^p(x)), & p > i > j \\
  L_{i,j}^{p-1}(x,\rho_i^p(x)),
  & i < p < j \\
  \resizebox{.42\columnwidth}{!}{$\frac{1}{\mu_j(0)} \int\limits_0^1
  K_i^m(x,0,y)\lambda(0,y)Q_j(y)dy$}, & p = m, i \leq j \\
  L_{i,j}^{p+1}(x,\rho_i^{p+1}(x)), & i \leq j \leq p < m
\end{cases}, \label{eq:bctmp2}
\end{align}
\end{subequations}
denote the boundary conditions according to the
terminal conditions of the characteristic projections solved in Lemma~\ref{lem:ccc}. The integral 
form \eqref{eq:infmKie} of the kernel
equations can then be employed in constructing the series of
successive approximations, by first inserting (arbitrary) initial
guesses for $K_i^p$ and $L_{i,j}^p$. The convergence of such successive approximations is 
established in Lemma~\ref{lem:sa}.

\begin{lemma}[Convergence of successive approximations]
\label{lem:sa}
Let  $j = 1,\ldots,m$ and $1 \leq i \leq p \leq m$ be arbitrary, and denote the sequences of 
successive approximations for the respective kernels $K_i^p$ and $L_{i,j}^p$ corresponding to 
\eqref{eq:infmKie}, \eqref{eq:bctmp} by 
$\left(K_{\ell}\right)_{\ell=0}^{\infty}$ and $\left( L_{\ell}\right)_{\ell=0}^{\infty}$, respectively, 
where we initialize $K_0$ and $L_0$ to zero. Then, the sequences of successive approximations 
converge such that 
\begin{subequations}
	\label{eq:saconv}
	\begin{align}
		\lim_{\ell \to \infty}\max_{(x,\xi) \in \mathcal{T}_i^p} \left\| K_\ell(x,\xi,\cdot) - 
		K_i^p(x,\xi,\cdot)\right\|_{L^2} & = 0, \\
		\lim_{\ell \to \infty}\max_{(x,\xi) \in \mathcal{T}_i^p}\left| L_\ell(x,\xi) - 
		L_{i,j}^p(x,\xi) \right| & = 
		0.
	\end{align}
\end{subequations}
\begin{proof}
Denote the differences of successive approximations by $\Delta K_\ell = K_{\ell} - K_{\ell-1}$ and 
$\Delta L_\ell = L_{\ell} - L_{\ell-1}$ for $\ell \geq 1$. As $K_0$ and $L_0$ were initialized to zero, 
the terms in the  sequences of successive approximations for $\ell \geq 1$ can be written as 
\begin{equation}
	\label{eq:saser}
	K_\ell = \sum_{l=1}^\ell \Delta K_l, \qquad L_\ell = \sum_{l=1}^\ell \Delta L_l.
\end{equation}
Now, the statement of the lemma is equivalent to the convergence of the series of differences 
\eqref{eq:saser} in the stated sense, which follows by showing that 
$\Delta K_\ell$ and $\Delta L_\ell$, for any $\ell \geq 1$, satisfy
\begin{subequations}
	\label{eq:DKL}%
	\begin{align}
		\| \Delta K_{\ell}(x,\xi,\cdot)\|_{L^{2}}
		& \leq M \frac{M_{K,L}^\ell\left(x - (1-\epsilon)\xi
			\right)^{\ell}}{\ell!}, \label{eq:DKL1} \\ 
		\left| \Delta L_{\ell}(x,\xi)\right|
		& \leq M \frac{M_{K,L}^\ell \left(x - (1-\epsilon)\xi
			\right)^{\ell}}{\ell!}, \label{eq:DKL2}
	\end{align}
\end{subequations}
uniformly on any $\mathcal{T}_i^p$, where $M, M_{K,L} > 0$ are given by
\begin{subequations}
	\label{eq:MM}%
	\begin{align}
		M & = \resizebox{.8\columnwidth}{!}{$\displaystyle M_B
                    + \left(1 + M_Q^1\right)\max_{x,y \in  
				[0,1]}\max_{j = \left\{
				1,\ldots,m \right\}}
                    \frac{|\theta_j(x,y)|}{\lambda(x,y) + \mu_i(x)}$}, 
		\\
		M_{K,L} & = m(1+M_Q^1)\left(M_{\lambda}^1 + M_{\sigma} + M_{\theta}
		\right)M_\lambda^\epsilon \nonumber \\
		& \qquad   + m\left(M_{\mu}^1 + M_W + M_{\psi}\right)M_\mu^\epsilon \label{eq:MM2},
	\end{align}
\end{subequations}
where $M_B = \max \left\{M_B^1, M_B^2 \right\}$ with
\begin{subequations}
	\label{eq:MB}%
	\begin{align}
		M_B^1 & = \max_{1 \leq i \neq j \leq m}\max_{x\in [0,1]} \left|
		\frac{\psi_{i,j}(x)}{\mu_i(x)-\mu_j(x)} \right|, \label{eq:MB1} \\
		M_B^2 & = \max_{1 \leq j < i \leq m}\max_{\xi \in
                        [0,1]} \left| \ell_{i,j}^{(1)}(\xi) \right|, 
	\end{align}
\end{subequations}
$M_\sigma, M_W$ and $M_\lambda, M_\mu$ are given by \eqref{eq:Ms}, \eqref{eq:MW}, 
and \eqref{eq:Mlammu}, respectively,
\begin{subequations}
	\label{eq:pbds}%
	\begin{align}
		M_{\lambda}^1 & = \max_{x,y\in [0,1]} \left|\lambda_x(x,y)\right|, \\
		M_{\mu}^1 & =  \max_{j=\left\{ 1,\ldots,m
			\right\}}\max_{x \in [0,1]}\left| \mu_j'(x) \right|, \\ 
		M_{\theta} & = \sum_{j=1}^m\max_{x\in [0,1]}\|\theta_{j}(x,\cdot)\|_{L^2}, \\
		M_{\psi} & = \max_{x\in [0,1]}\| \pmb{\Psi}(x) \|_1, \\
		M_Q^1 & = \max_{j=\left\{ 1,\ldots,m \right\}}\max_{y\in [0,1]}\
		\frac{\lambda(0,y)}{\mu_j(0)}\|Q_j\|_{L^2},
	\end{align}
\end{subequations}
where the parameter $\epsilon$ is taken 
such that
\begin{equation}
	\label{eq:eps}
	0 < \epsilon < 1 - \max_{1 \leq j < i \leq m}  \max_{
		x,\xi\in [0,1]}\frac{\mu_i(x)}{\mu_j(\xi)},\footnote{Such $\epsilon$ exists by \eqref{eq:muass}.}
\end{equation}
and
\begin{subequations}
	\label{eq:Mlm}%
	\begin{align}
		M_{\lambda}^{\epsilon} & = \max_{i \in \left\{ 1,\ldots,m \right\}}  \max_{x,\xi,y \in [0,1]}
		\frac{1}{\mu_i(\xi)+(1-\epsilon)\lambda(x,y)}, \\
		M_{\mu}^{\epsilon} & = \max_{i,j \in \left\{ 1,\ldots,m \right\}}  \max_{x,\xi
			\in [0,1]} \frac{-\epsilon_{i,j}}{\mu_i(\xi) -
			(1-\epsilon)\mu_j(x)},
	\end{align}
	where $\epsilon_{i,j}$ is given in \eqref{eq:epsij}.
\end{subequations}

Due to linearity, the integral equations and boundary conditions 
for $\Delta K_\ell$ and $\Delta L_\ell$ are of the same form as \eqref{eq:infmKie} and  
\eqref{eq:bctmp}, but with $K$ and $L$ replaced by $\Delta K_\ell$ and $\Delta L_\ell$. 	
Hence, the estimates \eqref{eq:DKL} can be proved by induction based on \eqref{eq:infmKie} and  
\eqref{eq:bctmp}. Firstly, the constant $M$ (and the initialization of $K_0, L_0$ to zero) 
guarantees that the estimates \eqref{eq:DKL} are 
satisfied for $\ell = 1$, and for any arbitrary $\ell > 1$ we have \eqref{eq:DKL} by the induction 
assumption. To show that \eqref{eq:DKL} 
then holds for $\ell \to \ell +1$, we insert the estimates \eqref{eq:DKL}, \eqref{eq:pbds}, and 
\eqref{eq:lyapb}, into the 
integral equations for $\Delta K_\ell$ and $\Delta L_\ell$. The following estimates are key to the 
induction step, and can be proved analogously to \cite[Lem. 6.2]{HuLDiM16}, for
all $i,j=1,\ldots,m$, $p=i,\ldots,m$, and any $\ell \geq 1$
\begin{subequations}
	\label{eq:sijpest}%
	\begin{align}
		\int\limits_0^{s_{i,p}^f(y)} \left(\hat{x}_{i,p}(s,y) -
		(1-\epsilon)\hat{\xi}_{i,p}(s,y) \right)^\ell ds
		& \leq \nonumber \\
		M_{\lambda}^{\epsilon}\frac{\left(x - (1-\epsilon)\xi
			\right)^{\ell+1}}{\ell+1}, &
\end{align}
\begin{align}
		\int\limits_0^{s_{i,j,p}^f} \left(\hat{x}_{i,j,p}(s) -
		(1-\epsilon)\hat{\xi}_{i,j,p}(s) \right)^\ell ds
		& \leq \nonumber \\
		M_{\mu}^{\epsilon}\frac{\left(x - (1-\epsilon)\xi
			\right)^{\ell+1}}{\ell+1},
		&                       
	\end{align}
\end{subequations}
where $(x,\xi) \in \mathcal{T}_i^p$ is the (arbitrary) initial point
of the respective characteristic curve on the $x\xi$-plane. 

Using \eqref{eq:sijpest} together with 
\eqref{eq:DKL} and the induction assumption, the induction step follows after similar 
computations as in \cite[Sect. VI.C]{AllKrs25}, albeit some 
additional care is required  due to splitting the domain into the $\mathcal{T}_i^p$ segments, as 
some boundary conditions depend on $\Delta K_\ell$ and $\Delta L_\ell$, which are unknown. 
However, as the boundary condition for $\Delta K_\ell$ on every $\mathcal{T}_i^i$ is known (due 
to \eqref{eq:bctmp1}), we 
can solve \eqref{eq:infmKie1} first on every $\mathcal{T}_i^i$, and then utilize the obtained values 
to solve \eqref{eq:infmKie1} on $\mathcal{T}_i^{i+1}$, and so on, up to  
$\mathcal{T}_i^m$.\footnote{Such a process is described in more detail in  \cite[Sect. 
3.2]{DiMArg18}.} As the domain $\mathcal{T}$ is split into at most $m$ segments, we need to 
solve \eqref{eq:infmKie1} at most $m$ times over the different segments to compute the next 
successive approximation. Hence, an adequate value for $M_{K,L}$ corresponding to the 
estimate for $\Delta K_\ell$ 
would be $m(M_\lambda^1+M_\sigma+ M_\theta)M_\lambda^\epsilon$, which gives the first term 
of~\eqref{eq:MM2}.

Deriving the estimate for $\Delta L_\ell$ follows similar steps, where we again need to traverse 
through the segments $\mathcal{T}_i^p$ (depending also on $j$) to have known boundary 
conditions for the integral equation \eqref{eq:infmKie2}. That is, for all $i\neq j$, we begin from 
$\mathcal{T}_i^i$ with known boundary condition on $\xi = x$ or $x = 1$, and then utilize the 
continuity conditions in \eqref{eq:bctmp2} up to $\mathcal{T}_i^m$ if $i > j$, or up to 
$\mathcal{T}_i^{j-1}$ if $i < j$. 
For $i 
\leq j$, the remaining segments are reached by beginning from $\mathcal{T}_i^m$ with the 
boundary condition on $\xi=0$, and then utilizing the continuity conditions up to 
$\mathcal{T}_i^{j}$. As in the case of $\Delta K_\ell$, this results in 
having to solve \eqref{eq:infmKie2} at most $m$ times, which results in the last term of 
\eqref{eq:MM2}. Moreover, the boundary condition on $\xi = 0$ depends on $\Delta K_\ell$, which 
can be dealt with using the estimate derived in the previous paragraph, which results in the 
remaining term $mM_Q^1(M_\lambda^1+M_\sigma+M_\theta)M_\lambda^\epsilon$ in 
\eqref{eq:MM2}. Thus, the estimate \eqref{eq:DKL} follows by induction. Hence the series 
\eqref{eq:saser} and, equivalently, the sequences of successive approximations converge in the 
stated sense \eqref{eq:saconv}.
\end{proof}
\end{lemma} 

\noindent {\itshape Proof of Theorem~\ref{thm:infmkwp}.} By Lemma~\ref{lem:sa}, the sequences 
of successive approximations for any $K_i^p$ and 
$L_{i,j}^p$ converge uniformly on $\mathcal{T}_i^p$ ($K_i^p$ in the
$L^2$ sense in $y$), which shows the existence (and well-posedness) of
the solutions $K_i^p, L_{i,j}^p$ to the kernel equations
\eqref{eq:infmkLk}--\eqref{eq:infmkbcc}.
To conclude the proof of Theorem \ref{thm:infmkwp},  
we note that any two $\mathcal{T}_i^p$ and $\mathcal{T}_i^s$ 
with $p\neq s$ may only intersect along a common boundary $\xi = \rho_j^r(x)$ for
$r=p$ or $r=s$ (if the segments are adjacent), which is a measure zero
subset of $\mathcal{T}$. Thus, as the kernels $L_{i,j}^p$ and $K_i^p$ are
continuous on each $(x,\xi) \in \mathcal{T}_i^p$, and the intersections of these segments 
comprise a finite number of sets of measure zero, the discontinuities of the
kernels $K_{i}$ and $L_{i,j}$ may only occur in sets of measure
zero.\footnote{In fact, the discontinuities may only occur in the
$L_{i,j}$ kernels for $1\leq i<j \leq  m$ along the curves $\xi =
\rho_i^j(x)$ due to \eqref{eq:infmkbcc}.} In particular, the kernels $K_i,L_{i,j}$ solving
\eqref{eq:infmk}--\eqref{eq:infmkabc} are uniquely determined by
$K_i^p, L_{i,j}^p$, almost everywhere on $\mathcal{T}$ and $\mathcal{P}$.

\section{Stabilization of Large-Scale $n+m$ Systems by Continuum
  Kernels} \label{sec:nmstab}

In this section, we construct a stabilizing control law for
large-scale $n+m$ systems based on the $\infty+m$ continuum
kernels. The core idea is to establish that, for large $n$, the exact control kernels constructed 
based on the $n+m$ system (see \ref{app:nm}) can be approximated to any desired accuracy by 
the continuum kernels computed on the basis of the $\infty+m$ system.

\subsection{Large-Scale $n+m$ Systems of Hyperbolic PDEs}

Consider a system of $n+m$ hyperbolic PDEs\footnote{We scale the sums involving the  $n$-part 
states $u^i, i=1,\ldots,n$ by $1/n$ in order 
to make the considerations in the limit $n\to\infty$ more natural, as discussed in \cite[Rem. 
2.2]{HumBek24arxiv}.}
\begin{subequations}
\label{eq:nmm}%
\begin{align}
\mathbf{u}_t(t,x) + \pmb{\Lambda}(x)\mathbf{u}_x(t,x)   & 
 	= \frac{1}{n}\pmb{\Sigma}(x)\mathbf{u}(t,x) +
	\mathbf{W}(x)\mathbf{v}(t,x), \\
	\mathbf{v}_t(t,x) - \mathbf{M}(x)\mathbf{v}_x(t,x)
	& = 
	\frac{1}{n}\pmb{\Theta}(x)\mathbf{u}(t,x) + \pmb{\Psi}(x)\mathbf{v}(t,x), 
	\end{align}
\end{subequations}
with boundary conditions
\begin{align}
\label{eq:nmmbc}%
\mathbf{u}(t,0) & = \mathbf{Q}\mathbf{v}(t,0), 
& & \mathbf{v}(t,1) = \mathbf{U}(t),
\end{align}
where $\pmb{\Lambda}(x) = 
\operatorname{diag}(\lambda_1(x),\ldots,\lambda_n(x)), 
\pmb{\Sigma}(x) = (\sigma_{i,j}(x))_{i,j=1}^n$, $\mathbf{W}(x) = (w_{i,j}(x))_{i=1,}^n{}_{j=1}^m,
\pmb{\Theta}(x) = 
(\theta_{j,i}(x))_{j=1,}^m{}_{i=1}^n, \mathbf{Q} = (q_{i,j})_{i=1,}^n{}_{j=1}^m, \mathbf{u} = 
\left(u^i\right)_{i=1}^n$, and $\mathbf{M}, 
\pmb{\Psi}$ correspond to the respective parameters in \eqref{eq:infm2}. As in 
\cite{AnfAamBook, HuLDiM16, HuLVaz19, AurDiM16}, we make the following assumptions
on the parameters.

\begin{assumption}
\label{ass:nm}
The transport velocities in \eqref{eq:nmm} are 
continuously differentiable with $\mu_m(x) > 0$ and $\lambda_i(x) > 0$ for all $x \in 
[0,1]$ and $i = 1,\ldots,n$, satisfying \eqref{eq:muass}.
The parameters $\pmb{\Sigma}, \mathbf{W}, \pmb{\Theta}, \pmb{\Psi}$ are
continuous with $\psi_{j,j} = 0$ for
all $j=1,\ldots,m$.
\end{assumption}

\begin{remark}
\label{rem:nmwp}
Under Assumption~\ref{ass:nm}, it can be shown by using the same arguments as in \cite[Prop. 
A.1]{HumBek24arxiv} that the system \eqref{eq:nmm}, \eqref{eq:nmmbc} is well-posed on the 
Hilbert space $E$.
\end{remark}

\subsection{Control Design and Stability Result}

Consider any continuous 
functions $\theta_j, W_j, Q_j, \lambda$, and $\sigma$ satisfying Assumption~\ref{ass:infm} with
\begin{subequations}
\label{eq:nmcap}%
\begin{align}
\theta_j(x,i/n) & = \theta_{j,i}(x), \\
W_j(x,i/n) & = w_{i,j}(x), \\
Q_j(i/n) & = q_{i,j}, \\
\lambda(x,i/n) & = \lambda_i(x), \\
\sigma(x,i/n,l/n) & = \sigma_{i,l}(x), \label{eq:nmcape}
\end{align}
\end{subequations}
for all $x \in [0,1]$, $i,l = 1,\ldots,n$ and $j = 1,\dots,m$.\footnote{As noted in \cite[Sect. 
IV.A]{HumBek24arxiv}, there are infinitely 
many ways to construct continuous (or even smooth in $y$) functions satisfying \eqref{eq:nmcap} 
and Assumption~\ref{ass:infm} based on parameters satisfying Assumption~\ref{ass:nm}.} It 
follows by 
Theorem~\ref{thm:infmkwp} that the corresponding continuum 
kernel equations \eqref{eq:infmkLk}--\eqref{eq:infmkbcc} have well-posed solution $K_i^p \in 
C(\mathcal{T}_i^p; L^2([0,1];\mathbb{R}))$, $L_{i,j}^p \in C(\mathcal{T}_i^p; \mathbb{R})$ for 
all $j=1,\ldots,m$ and $1 \leq i \leq p \leq m$. Thus, construct the following functions for all 
$(x,\xi) \in \mathcal{T}_i^p$ with $1 \leq i \leq p \leq m$,\footnote{If 
	$K_i^p(x,\xi,\cdot)$ is continuous, the mean-value sampling in \eqref{eq:kerapp1} can be 
	replaced with pointwise evaluation, e.g., at $y=1/n,\ldots,n$.\label{fn:kerapp}} 
\begin{subequations}
\label{eq:kerapp}
\begin{align}
\widetilde{k}^p_{i,l}(x,\xi) & = n\int\limits_{(l-1)/n}^{l/n}K_i^p(x,\xi, y)dy, \quad l = 1, \ldots 
n, \label{eq:kerapp1} \\
\widetilde{\ell}_{i,j}^p(x,\xi) & = L_{i,j}^p(x, \xi), \quad j = 1, \dots, m.
\end{align}
\end{subequations}
We have the following stabilization result.
\begin{theorem}
\label{thm:nmstab}
Consider \eqref{eq:nmm}, \eqref{eq:nmmbc} satisfying Assumption~\ref{ass:nm}. Let (continuum) 
parameters $\theta_j, W_j, Q_j$ for $j=1,\ldots,m$ and $\sigma,\lambda$ satisfy 
Assumption~\ref{ass:infm} and relations 
\eqref{eq:nmcap}. Define the feedback laws
\begin{align}
\label{eq:Un}
U^i(t) & =  \sum_{l=1}^n\sum_{p=i}^m
\int\limits_{\rho_i^{p+1}(1)}^{\rho_i^p(1)}\frac{1}{n}
\widetilde{k}^{p}_{i,l}(1,\xi)u^l(t,\xi)d\xi \nonumber \\
& \qquad + \sum_{j=1}^m\sum_{p=i}^m
\int\limits_{\rho_i^{p+1}(1)}^{\rho_i^p(1)} \widetilde{\ell}^{p}_{i,j}(1,\xi)v^j(t,\xi)d\xi,
\end{align}
for $i = 1,\ldots,m$, where $\rho_i^p$ are given in \eqref{eq:rho} and 
$\widetilde{k}_{i,l}^p,\widetilde{\ell}_{i,j}^p$ are given by 
\eqref{eq:kerapp} for all $j = 1,\ldots,m$ and
$1 \leq i \leq p \leq m$. The control law \eqref{eq:Un} exponentially 
stabilizes system \eqref{eq:nmm},~\eqref{eq:nmmbc} on $E$, provided that $n$ is sufficiently 
large.
\end{theorem}

\subsection{Proof of Theorem~\ref{thm:nmstab}}

The proof of Theorem~\ref{thm:nmstab} is presented at the end of this
section based on the following lemmas.

\begin{lemma}[Transforming $n+m$ kernels from $E$ to $E_c$]
  \label{lem:pck}
Consider the $n+m$ kernel equations \eqref{eq:nmk}--\eqref{eq:nmkbcc} with parameters 
satisfying Assumption~\ref{ass:nm} and define the following functions for all $x \in [0,1]$, 
piecewise in $y$ for $i,l = 1,\ldots,n$ and $j = 1,\ldots,m$
\begin{subequations}
\label{eq:infmppc}%
\begin{align}
\lambda^n(x,y) & = \lambda_i(x), \quad y \in
(i-1)/n, i/n], \\
\sigma^n(x,y,\eta) & = \sigma_{i,l}(x), \quad y \in
((i-1)/n,i/n],\\
& \qquad \qquad \qquad \eta \in ((l-1)/n, l/n], \\
W_j^n(x,y) & = w_{i,j}(x), \quad y \in ((i-1)/n, i/n], \\
\theta_j^n(x,y) & = \theta_{j,i}(x), \quad y \in
((i-1)/n,i/n], \\
Q_j^n(y) & = q_{i,j}, \quad y \in ((i-1)/n, i/n].
\end{align}
\end{subequations}
Construct the following functions for $(x,\xi) \in \mathcal{T}_i^p$
with $1 \leq i \leq p \leq m$,  
piecewise in $y$ for $l = 1,\ldots,n$
\begin{equation}
	\label{eq:kn}
	K_{i,p}^n(x,\xi,y) = k_{i,l}^p(x,\xi), \quad y \in
	(l-1)/n, l/n], 
\end{equation}
where $k_{i,l}^p$ is the solution to \eqref{eq:nmkk} on $\mathcal{T}_i^p$. Then, 
$K_{i,p}^n$ together with $\ell_{i,j}^p$ for $j=1,\ldots,m$ (the solution to \eqref{eq:nmkl} 
on $\mathcal{T}_i^p$) satisfy the continuum kernel equations \eqref{eq:infmkLk}, 
\eqref{eq:infmkLkbc}--\eqref{eq:infmkbcc} with parameters defined in 
\eqref{eq:infmppc} and the original $\mu_j, \psi_{i,j}, l_{i,j}$, for $i,j = 1,\ldots,m$.
\end{lemma}
\begin{proof}
We define the linear transform $\mathcal{F} =
\operatorname{diag}(\mathcal{F}_{n}, I_m)$ where 
$\mathcal{F}_n\mathbf{e}_{i} = \chi_{((i-1)/n,i/n]}$ with
$\chi_{((i-1)/n,i/n]}$ being the indicator function of the
interval $((i-1)/n,i/n]$ and
$\left(\mathbf{e}_{i}\right)_{i=1}^n$ being the Euclidean basis
of $\mathbb{R}^n$. Thus, the transform maps any $\mathbf{b}
=\left(b_i\right)_{i=1}^{n+m} \in \mathbb{R}^{n+m}$ into
$L^2\left([0,1];\mathbb{R}\right)\times \mathbb{R}^m$ as
\begin{equation}
\mathcal{F}\mathbf{b} = \begin{bmatrix}
\sum_{i=1}^n b_i \chi_{((i-1)/n,i/n]} \\ (b_j)_{j=n+1}^{n+m}.
\end{bmatrix}.
\end{equation}
For any $g \in L^2([0,1];\mathbb{R})$,
the adjoint $\mathcal{F}_n^{*}$ satisfies
\begin{equation}
\left\langle \mathcal{F}_n\mathbf{e}_\ell, g
\right\rangle_{L^2([0,1];\mathbb{R})}
= \int\limits_{(\ell-1)/n}^{\ell/n}g(y)dy =
\frac{1}{n}\left\langle \mathbf{e}_{\ell}, \mathcal{F}_n^{*}g \right\rangle_{\mathbb{R}^{n}},
\end{equation}
that is, $\mathcal{F}_n^{*}$ is given by
\begin{equation}
\label{eq:Fns}
\mathcal{F}_n^{*}g = \left( n\int\limits_{(i-1)/n}^{i/n}g(y)dy \right)_{i=1}^n,
\end{equation}
where each component is the mean value of $g$ over the interval $[(i-1)/n,i/n]$.
Thus,  $\mathcal{F}$ has the adjoint
$\mathcal{F}^{*} = \operatorname{diag}\left(\mathcal{F}_n^{*},I_m\right)$, which
additionally satisfies $\mathcal{F}^{*}\mathcal{F} = I_{n+m}$, i.e., $\mathcal{F}$ (and 
$\mathcal{F}_n$) are isometries, and thus, norm preserving from their domain to their 
co-domain. Now, the claim follows similarly to \cite[Lem. 4.2]{HumBek24arxiv} after applying 
$\mathcal{F}$ to \eqref{eq:nmk} from the left (pointwise in $(x,\xi) \in \mathcal{T}_i^p$ for every 
$1\leq i \leq p \leq m$), using the fact that
$\mathcal{F}^{*}\mathcal{F} = I_{n+m}$ and the definitions \eqref{eq:infmppc}, 
\eqref{eq:kn}. Moreover, one needs to verify that the boundary conditions 
\eqref{eq:infmkLkbc}--\eqref{eq:infmkbcc} are satisfied, which is trivially true for 
\eqref{eq:infmkLkbca}, \eqref{eq:abct}, \eqref{eq:infmkLkcbca} as the $\mathcal{F}$ 
transform is identity in the second component, and the remaining \eqref{eq:infmkLkbcb}, 
\eqref{eq:infmkLkbcc}, \eqref{eq:infmkLkcbcb} are verified utilizing the fact that 
$K_{i,p}^n$ are piecewise constant in $y$.
\end{proof}

\begin{lemma}[Approximating $n+m$ kernels by continuum]
\label{lem:nmka}
Consider the solutions $K^n_{i,p}, \ell^p_{i,j}$ for $j = 1,\ldots,m, 1 \leq i \leq p \leq m$ to the 
kernel equations \eqref{eq:infmkLk},
\eqref{eq:infmkLkbc}--\eqref{eq:infmkbcc} for any $n\in\mathbb{N}$ with parameters 
$\lambda^n, \mu_j,
\sigma^n, W^n_j, \theta_j^n, Q_j^n, \psi_{i,j}, l_{i,j}$, for $i,j =
1,\ldots,m$, from Lemma~\ref{lem:pck}. There exist continuum parameters $\lambda,
\sigma, W_j, \theta_j, Q_j$ constructed such that they
satisfy \eqref{eq:nmcap} and together with the original parameters $\mu_j,
\psi_{i,j}$ they satisfy Assumption~\ref{ass:infm}. For any such
parameters, the solution $K_i^p, L_{i,j}^p$ for $j = 1,\ldots,m, 1 \leq i \leq p \leq m$ to the 
continuum kernel equations \eqref{eq:infmkLk},
\eqref{eq:infmkLkbc}--\eqref{eq:infmkbcc}, where $l_{i,j}^{(1)} =
l_{i,j}$, exists and satisfies the following implications. For any
$\delta_1 > 0$, there exists an $n_{\delta_1} \in \mathbb{N}$ such
that for all $n \geq n_{\delta_1}$ we have 
\begin{subequations}
  \label{eq:nmka}%
\begin{align}
	\max_{1\leq i \leq p \leq m}\max_{(x,\xi) \in \mathcal{T}_i^p}\|K_i^p(x,\xi,\cdot) -
  K^n_{i,p}(x,\xi,\cdot)\|_{L^2} & \leq \delta_1,\label{eq:nmka1} \\
 	 \max_{j= \{ 1,\ldots,m\}} \max_{1\leq i \leq p \leq m}\max_{(x,\xi) \in 
 	\mathcal{T}_i^p}|L_{i,j}^p(x,\xi) -
  \ell_{i,j}^p(x,\xi)| & \leq \delta_1. \label{eq:nmka2}
\end{align}
\end{subequations}
\begin{proof}
Following the same steps as in the proof of
\cite[Lem. 4.3]{HumBek24arxiv}, we first note that the the kernel
equations \eqref{eq:infmkLk},
\eqref{eq:infmkLkbc}--\eqref{eq:infmkbcc} have well-posed solutions
for the two sets of parameters considered in the statement of the
lemma. Thus, $K_{i,p}^n, \ell_{i,j}^p$ depend continuously on
$\lambda^n, \mu_j, \sigma^n, W_j^n, \theta_j^n, Q_j^n$, $\psi_{i,j},$
and $l_{i,j}$ by \cite[Thm A.1]{HuLVaz19} and Lemma~\ref{lem:pck}, and $K_i^p, L_{i,j}^p$ 
depend continuously on $\lambda,
\mu_j, \sigma, W_j, \theta_j$, $Q_j, \psi_{i,j}$, and $l_{i,j}$ by Theorem~\ref{thm:infmkwp}, for $i,j
= 1,\ldots, m$ with $i \leq p \leq m$. As the parameters $\mu_j,
\psi_{i,j}, l_{i,j}$ coincide, the claim follows after showing that
the parameters $\lambda^n, \sigma^n, W_j^n, \theta_j^n, Q_j^n$
converge to $\lambda, \sigma, W_j, \theta_j,Q_j$ (for all
$j=1,\ldots,m$) as $n\to \infty$. This convergence is established under \eqref{eq:infmppc} by 
\cite[Sect. 1.3.5]{TaoBook11} in the sense
that, for any $\varepsilon_1 > 0$, the following estimates are
satisfied for any sufficiently large $n$
\begin{subequations}
\label{eq:sa}%
\begin{align}
\max_{x\in[0,1]}\|\lambda(x,\cdot) -
  \lambda^n(x,\cdot)\|_{L^2([0,1];\mathbb{R})}
  & \leq 
\varepsilon_1, \\
\max_{x\in[0,1]}\|\sigma(x,\cdot) - 
\sigma^n(x,\cdot)\|_{L^2([0,1]^2;\mathbb{R})}
& \leq 	\varepsilon_1, \\
\max_{j=1,\ldots,m}\max_{x\in[0,1]}\|\theta_j(x,\cdot) -
  \theta_j^n(x,\cdot)\|_{L^2([0,1];\mathbb{R})}
  & \leq \varepsilon_1, \\
  \max_{j=1,\ldots,m}\max_{x\in[0,1]}\|W_j(x,\cdot) - W^n_j(x,\cdot)\|_{L^2([0,1];\mathbb{R})}
  & \leq \varepsilon_1, \\
  \max_{j=1,\ldots,m}\|Q_j-Q_j^n\|_{L^2([0,1];\mathbb{R})}
  & \leq \varepsilon_1.
\end{align}%
\end{subequations}
\end{proof}
\end{lemma}
\begin{remark}
As noted in \cite[Rem. 4.4]{HumBek24arxiv}, the step functions
$\lambda^n, \sigma^n, \theta_j^n, W_j^n, Q_j^n$ could be constructed
in various, alternative ways to \eqref{eq:infmppc}, but the result of
Lemma~\ref{lem:nmka} remains valid as long as the step functions
approximate the continuous functions $\lambda,\sigma,\theta_j,W_j,Q_j$ (satisfying 
\eqref{eq:nmcap}) to arbitrary accuracy as $n\to \infty$ as in \eqref{eq:sa}.
\end{remark}

\begin{lemma}[Representation of \eqref{eq:Un} as perturbation of exact controller]
  \label{lem:Und}
  The control law \eqref{eq:Un} can be written as
\begin{align}
\label{eq:Und}
    U^i(t)
    & =  \sum_{l=1}^n\sum_{p=i}^m
\int\limits_{\rho_i^{p+1}(1)}^{\rho_i^p(1)}\frac{1}{n}
k^{p}_{i,l}(1,\xi)u^l(t,\xi)d\xi \nonumber \\
& \qquad + \sum_{j=1}^m\sum_{p=i}^m
\int\limits_{\rho_i^{p+1}(1)}^{\rho_i^p(1)}
\ell^{p}_{i,j}(1,\xi)v^j(t,\xi)d\xi
 \nonumber \\
   &  + \sum_{l=1}^n\sum_{p=i}^m
\int\limits_{\rho_i^{p+1}(1)}^{\rho_i^p(1)}\frac{1}{n}
\Delta k^{p}_{i,l}(1,\xi)u^l(t,\xi)d\xi \nonumber
 \end{align}
\begin{align}
& \qquad + \sum_{j=1}^m\sum_{p=i}^m
\int\limits_{\rho_i^{p+1}(1)}^{\rho_i^p(1)} \Delta \ell^{p}_{i,j}(1,\xi)v^j(t,\xi)d\xi,
\end{align}
where $k^p_{i,l}, \ell^p_{i,j}$ is the (exact) solution to the $n+m$ kernel equations 
\eqref{eq:nmk}--\eqref{eq:nmkbcc} and $\Delta k_{i,l}^p, \Delta 
\ell_{i,j}^p$ are the approximation error terms that become arbitrarily small, uniformly in $\xi \in 
[\rho_i^{p+1}(1), \rho_i^{p}(1)]$, for all 
$l = 1,\ldots,n$ and $i,j = 1,\ldots,m$ with $i \leq p \leq m$, when $n$ is sufficiently large.
\begin{proof}
Transform the functions $\widetilde{k}_{i,l}^p$ from \eqref{eq:kerapp} into step functions in $y$ 
as 
\begin{equation}
	\label{eq:ktapp}
	\widetilde{K}_{i,p}^n(x,\xi,y) = \widetilde{k}_{i,l}^p(x,\xi), \quad y \in ((l-1)/n, l/n],
\end{equation}
for all $(x,\xi) \in \mathcal{T}_i^p$ and $1 \leq i \leq p \leq m$, piecewise for $y$ for $l = 
1,2,\ldots,n$. By \eqref{eq:kerapp} and \eqref{eq:ktapp}, we have $	
\widetilde{K}_{i,p}^n(x,\xi,\cdot) = \mathcal{F}_n\mathcal{F}_n^*K_{i}^p(x,\xi,\cdot)$, which is the 
mean-value approximation of $K_{i}^p(x,\xi,\cdot)$ for all $(x,\xi) \in \mathcal{T}_i^p$ and $1 
\leq i \leq p 
\leq m$. By \cite[Sect. 1.6]{TaoBook11}, the mean-value approximation becomes arbitrarily 
accurate for sufficiently large $n$, i.e., for any $\varepsilon_2 > 0$ 
there exists some $n_{\varepsilon_2} \in \mathbb{N}$ such that
\begin{equation}
	\label{eq:kmv}
	\max_{1\leq i \leq p \leq m}\max_{(x,\xi)\in \mathcal{T}_i^p}\|K_i^p(x,\xi,\cdot) - 
	\widetilde{K}_{i,p}^n(x,\xi,\cdot)\|_{L^2} \leq \varepsilon_2,
\end{equation}
for any $n \geq n_{\varepsilon_2}$. Combining \eqref{eq:kmv} with the estimate \eqref{eq:nmka1}
and using the triangle inequality, we have for any $n \geq \max \left\{n_{\delta_1}, 
n_{\varepsilon_2}\right\}$
\begin{align}
	\label{eq:ktkn}
		\max_{1\leq i \leq p \leq m}\max_{(x,\xi)\in \mathcal{T}_i^p}\|K_{i,p}^n(x,\xi,\cdot) - 
	\widetilde{K}_{i,p}^n(x,\xi,\cdot)\|_{L^2}
	& \leq  \nonumber \\
		\max_{1\leq i \leq p \leq m}\max_{(x,\xi) \in \mathcal{T}_i^p}\|K_i^p(x,\xi,\cdot) -
	K^n_{i,p}(x,\xi,\cdot)\|_{L^2} & \nonumber \\
	+ 	\max_{1\leq i \leq p \leq m}\max_{(x,\xi)\in \mathcal{T}_i^p}\|K_i^p(x,\xi,\cdot) - 
	\widetilde{K}_{i,p}^n(x,\xi,\cdot)\|_{L^2} & \leq \nonumber \\
	\delta_1 + \varepsilon_2,
\end{align}
where both $\delta_1$ and $\varepsilon_2$ can be made arbitrarily small by increasing $n$, 
which follows from 
Lemma~\ref{lem:nmka} and \eqref{eq:ktapp}, respectively. As the estimate is uniform on every 
$\mathcal{T}_i^p$, it particularly applies on $x=1$.

Moreover, the step functions $\widetilde{K}_{i,p}^n$ and $K_{i,p}^n$ constructed in 
\eqref{eq:ktapp} and \eqref{eq:kn}, respectively, are obtained through applying the transform 
$\mathcal{F}_n$, introduced in the proof of Lemma~\ref{lem:pck}, to 
$\left(\widetilde{k}^{p}_{i,l}\right)_{l=1}^n$ and $\left(k^p_{i,l}\right)_{l=1}^n$, respectively, for all 
$i=1,\ldots,m$. Thus, as $\mathcal{F}_n$ is an isometry, the estimate \eqref{eq:ktkn} also holds 
for these functions, i.e., 
\begin{align}
	\label{eq:kpd}
	\resizebox{.95\columnwidth}{!}{$\displaystyle
\max_{1\leq i \leq p \leq m}\max_{(x,\xi)\in 
\mathcal{T}_i^p}\frac{1}{\sqrt{n}}\left\|\left(k_{i,l}^p(x,\xi)\right)_{l=1}^n - 
\left(\widetilde{k}^{p}_{i,l}(x,\xi)\right)_{l=1}^n
\right\|_{\mathbb{R}^n}$} & \leq \nonumber \\
\delta_1 +  \varepsilon_2. &
\end{align}
In addition, from \eqref{eq:nmka2} in Lemma~\ref{lem:nmka} we already have for $n \geq 
n_{\delta_1}$
\begin{equation}
		\label{eq:lpd}
		\resizebox{\columnwidth}{!}{$\displaystyle \max_{1\leq i \leq p \leq m} \max_{(x,\xi) \in 
		\mathcal{T}_i^p}\left\|\left(\ell_{i,j}^p(x,\xi)\right)_{j=1}^m -
	\left(\widetilde{\ell}^p_{i,j}(x,\xi)\right)_{j=1}^m\right\|_{\mathbb{R}^m} \leq \sqrt{m}\delta_1.$}
\end{equation}
Now, setting $\Delta k_{i,l}^p = \widetilde{k}_{i,l}^p - k_{i,l}^p$ and $\Delta \ell_{i,j}^p = 
\widetilde{\ell}_{i,j}^p - \ell_{i,j}^p$, we have written \eqref{eq:Un} as \eqref{eq:Und}, where the 
error term can be estimated using \eqref{eq:kpd}, \eqref{eq:lpd}, triangle inequality, and 
Cauchy-Schwartz inequality, for all $i=1,\ldots,m$, as
\begin{align}
	\label{eq:Udest}
	\sum_{l=1}^n\sum_{p=i}^m
	\int\limits_{\rho_i^{p+1}(1)}^{\rho_i^p(1)}\frac{1}{n}
	\Delta k^{p}_{i,l}(1,\xi)u^l(t,\xi)d\xi &  \nonumber \\
	+ \sum_{j=1}^m\sum_{p=i}^m
	\int\limits_{\rho_i^{p+1}(1)}^{\rho_i^p(1)} \Delta \ell^{p}_{i,j}(1,\xi)v^j(t,\xi)d\xi & \leq \nonumber \\
	(\delta_1 + \varepsilon_2  + \sqrt{m}\delta_1)\left\|\left(\begin{smallmatrix}
		\mathbf{u}(t,\cdot) \\ \mathbf{v}(t,\cdot)
	\end{smallmatrix}\right)\right\|_E, &
\end{align}
where $\delta_1$  and $\varepsilon_2$ become arbitrarily small when $n$ is sufficiently 
large.
\end{proof}
\end{lemma}

{\itshape Proof of Theorem~\ref{thm:nmstab}.} By Lemma~\ref{lem:Und}, the control law 
\eqref{eq:Un} splits into two parts as shown in \eqref{eq:Und}, where the first part exponentially 
stabilizes the system \eqref{eq:nmm}, \eqref{eq:nmmbc} by \cite[Thm 2.1]{HuLVaz19}, while the 
second one becomes arbitrarily
small when $n$ is arbitrarily large. Thus, the exponential stability of the closed-loop system 
\eqref{eq:nmm}, 
\eqref{eq:nmmbc} 
under control law \eqref{eq:Und}, when $n$ is sufficiently large, follows directly from the 
well-posedness of \eqref{eq:nmm}, \eqref{eq:nmmbc} (see Remark~\ref{rem:nmwp}) and 
\cite[Prop. A.2]{HumBek24arxiv}.

\begin{remark}
In order to quantify an upper bound on $\delta_1$ and $\varepsilon_2$ in \eqref{eq:Udest}, one 
can employ Lyapunov-based arguments similarly to \cite[Sect. IV.C]{HumBek24arxiv}. A 
Lyapunov functional for 
\eqref{eq:nmm}, \eqref{eq:nmmbc} under the control law \eqref{eq:Und} can be constructed 
similarly to \eqref{eq:lyap} for $\mathbf{u}(t,\cdot) \in 
L^2([0,1];\mathbb{R}^n)$ (see also 
\cite[Prop. 2.1]{HuLVaz19}), i.e.,
\begin{equation}
\label{eq:lyapn}
\widetilde{V}(t) = \int\limits_0^1 \left(e^{-\tilde{\delta}
x}\|\mathbf{u}(t,x)\|_{\pmb{\Lambda^{-1}}}^2 
+ e^{\tilde{\delta} 
x}\|\pmb{\beta}(t,x)\|_{\widetilde{\mathbf{D}}\mathbf{M}^{-1}}^2\right)dx,
\end{equation}
where $\pmb{\beta}$ is defined  for $\mathbf{u}(t,\cdot)\in L^2([0,1];\mathbb{R}^n)$ in 
\cite[(28)]{HuLDiM16} and satisfies \eqref{eq:infmts2}, \eqref{eq:infmtsbc2}. The stability analysis 
follows the same 
steps
as in the proof of 
Theorem~\ref{thm:stab}, which results~in
\begin{equation}
	\label{eq:Vnest}
	\dot{\widetilde{V}}(t) \leq - \frac{\tilde{c}_V}{\max\left\{\widetilde{M}_\mu, 
	\widetilde{M}_\lambda\right\}}\widetilde{V}(t) 
	+ e^{\tilde{\delta}x}\|\pmb{\beta}(1,t)\|_{\widetilde{\mathbf{D}}}^2,
\end{equation}
for some $\tilde{c}_V, \widetilde{M}_\mu, \widetilde{M}_\lambda > 0$, where the boundary 
condition in $\pmb{\beta}(t,1)$ is 
not zero 
but it is given by the error terms in the 
control law \eqref{eq:Und}. Given the estimate \eqref{eq:Udest} and since by \cite[Thm. 
3.4]{HuLDiM16} the $n+m$
backstepping transformation (of $\pmb{\beta}$) is invertible, we can estimate
from \eqref{eq:Und}
\begin{equation}
	\label{eq:best}
	\|\pmb{\beta}(t,1)\|^2_{\widetilde{\mathbf{D}}} \leq \left\|\widetilde{\mathbf{D}}\right\|_\infty 
	m^2(\delta_1+\varepsilon_2+\sqrt{m}\delta_1)^2M_V^2\left\|\left(\begin{smallmatrix}
		\mathbf{u}(t,\cdot) \\ \pmb{\beta}(t,\cdot)
	\end{smallmatrix}\right)\right\|_E^2,
\end{equation}
where $M_V$ is a bound on the inverse backstepping transformation of $\pmb{\beta}$ for the 
$n+m$ 
case (see \cite[(45)]{HuLDiM16}). Now, if 
$\tilde{\delta}, \widetilde{\mathbf{D}} > 0$ have been fixed such that \eqref{eq:Vnest} holds for 
$\pmb{\beta}(1,t) =0$,  for $\delta_1,\varepsilon_2 > 0$ sufficiently small, given a 
sufficiently large $n$, $\dot{\widetilde{V}}$ 
remains negative definite also when $\pmb{\beta}(1,t)$ is estimated with \eqref{eq:best}.
\end{remark}

\section{Convergence of the Large-Scale System to a Continuum} \label{sec:nmapp}

While in Section~\ref{sec:nmstab} we present an approximation result that concerns 
backstepping kernels, in the present section we provide a formal proof that the actual solutions of 
the PDE system \eqref{eq:nmm}, \eqref{eq:nmmbc} converge to the solutions of the continuum 
PDE system \eqref{eq:infm}, \eqref{eq:infmbc}, as $n\to\infty$. The following result is of interest 
itself as it provides a formal connection between the solutions of the $n+m$ system and the 
solutions of its continuum counterpart.
\begin{theorem}
	\label{thm:solappr}
Consider an $n+m$ system \eqref{eq:nmm}, \eqref{eq:nmmbc} with parameters $\mu_j, 
\psi_{j,\ell}, \theta_{j,i},w_{i,j}, q_{i,j},\lambda_i$, and $\sigma_{i,l}$ for $i,l = 1,\ldots$, $n$ and  
$j,\ell = 1,\ldots,m$, satisfying Assumption~\ref{ass:nm}, initial conditions $(\mathbf{u}_0, 
\mathbf{v}_0) \in E$, and input $\mathbf{U} \in L_{\rm loc}^2([0,+\infty); \mathbf{R}^m)$. 
Construct a continuum system \eqref{eq:infm}, \eqref{eq:infmbc} with parameters $\lambda, 
\mu_j, \sigma, \theta_j, W_j, Q_j, \psi_{j,\ell}$ for $j,\ell=1,\ldots,m$ that satisfy 
Assumption~\ref{ass:infm} and \eqref{eq:nmcap}, and equip \eqref{eq:infm}, \eqref{eq:infmbc} 
with initial conditions $u_0, \mathbf{v}_0$ and input $\mathbf{U}$, such that $u_0$ is continuous 
in $y$ and satisfies 
\begin{equation}
	\label{eq:uapp}
	u_0(x,i/n) = u_0^i(x), \quad i = 1,\ldots,n.
\end{equation}
Sample the solution $(u, \mathbf{v})$ to the resulting PDE system \eqref{eq:infm}, 
\eqref{eq:infmbc} for these data into a vector-valued function 
$(\widetilde{\mathbf{u}},\widetilde{\mathbf{v}})$ as $\widetilde{\mathbf{u}}(t,x) = 
\mathcal{F}_n^*u(t,x,\cdot)$ (see \eqref{eq:Fns}) and $\widetilde{\mathbf{v}}(t,x) = 
\mathbf{v}(t,x)$, pointwise for all $t\geq 0$ and almost all $x \in [0,1]$. On any compact interval 
$t \in [0,T]$, for any given $T > 0$, we have 
	\begin{equation}
	\label{eq:unest}
	\max_{t\in[0, T]}\left\| \left( 
	\begin{smallmatrix}
		\mathbf{u}(t) \\ \mathbf{v}(t)
	\end{smallmatrix}
	\right) -  \left( 
	\begin{smallmatrix}
		\widetilde{\mathbf{u}}(t) \\ \widetilde{\mathbf{v}}(t)
	\end{smallmatrix}
	\right) \right\|_E \leq \varepsilon_3,
\end{equation}
where $\varepsilon_3 > 0$ becomes arbitrarily small when $n$ is 
sufficiently large.
\begin{proof}
The statement follows applying the same steps as in the proof of \cite[Thm 6.1]{HumBek24arxiv}. 
In a nutshell, as 
the systems \eqref{eq:nmm}, \eqref{eq:nmmbc} and \eqref{eq:infm}, \eqref{eq:infmbc} have 
well-posed solutions under the assumptions of the theorem by Remark~\ref{rem:nmwp} and 
Remark~\ref{rem:infmwp}, respectively, the solutions in particular depend continuously on the 
parameters, initial conditions, and inputs of the respective PDEs. Now, as the input to 
\eqref{eq:nmm}, \eqref{eq:nmmbc} and \eqref{eq:infm}, \eqref{eq:infmbc} is the same, while the 
respective
parameters converge as in \eqref{eq:sa} and $u_0$ can be approximated to arbitrary 
accuracy by $\mathcal{F}_n \mathbf{u}_0$ due to \eqref{eq:uapp}, we have that $\displaystyle 
\mathcal{F}\left( 
\begin{smallmatrix}
	\mathbf{u}(t) \\ \mathbf{v}(t)
\end{smallmatrix}
\right)$ approximates $\left( 
\begin{smallmatrix}
	u(t) \\ \mathbf{v}(t)
\end{smallmatrix}
\right)$ arbitrary accuracy on $E_c$, when $n$ is sufficiently large. Since 
$\widetilde{\mathbf{u}}(t,x) = \mathcal{F}_n^*u(t,x,\cdot)$ and $\widetilde{\mathbf{v}}(t,x) = 
\mathbf{v}(t,x)$, while $\mathcal{F}_n$ (and $\mathcal{F}$) is an isometry and the 
mean-value approximation $\mathcal{F}_n\mathcal{F}_n^*u(t)$ of the solution $u(t)$ (which is 
bounded on $t \in[0,T]$) is convergent (in $L^2$; 
see \cite[Sect. 1.6]{TaoBook11}), the estimate \eqref{eq:unest} follows by the triangle inequality as 
in \cite[(62)]{HumBek24arxiv}.
\end{proof}
\end{theorem}

\section{Numerical Example and Simulation Results} \label{sec:ex}

Consider the parameters for $x,y,\eta \in [0,1]$
\begin{subequations}
\label{eq:exp}
	\begin{align}
		\lambda(x,y) & = 1, \quad \mu_1(x) = 2, \quad \mu_2(x) = 1 \\
		\sigma(x,y,\eta) & = x^3(x+1)\left(y - \frac{1}{2}\right)\left(\eta- \frac{1}{2}\right), \\
		W_1(x,y) & = W_2(x,y) = x(x+1)e^x\left(y - \frac{1}{2}\right), \\
		\theta_1(x,y) & = -3y(y-1), \quad
		\theta_2(x,y) = -2y(y-1), \\
		\psi_{i,j}(x) & = 0, \quad i,j \in \{1,2\}, \\
		Q_1(y) & = 8\left(y - \frac{1}{2}\right), \quad
		Q_2(y) = -8(y-2),
	\end{align}
\end{subequations}
corresponding to an $\infty+m$ system for $m=2$, which can be viewed as a continuum 
approximation of an 
$n+m$ system (for large $n$) based on \eqref{eq:nmcap} with respective parameters 
\begin{subequations}
	\label{eq:expnm}%
	\begin{align}
		\lambda_i(x) & = 1, \\
		 \sigma_{i,l}(x) & = x^3(x+1)\left(\frac{i}{n} - \frac{1}{2}\right)\left(\frac{l}{n}- \frac{1}{2}\right), 
		 \\
		\theta_{1,i}(x) & = -3\frac{i}{n}\left(\frac{i}{n}-1\right), \quad 	\theta_{2,i}(x) = 
		-2\frac{i}{n}\left(\frac{i}{n}-1\right), \\
		w_{i,1}(x) & = w_{i,2}(x) = x(x+1)e^x\left(\frac{i}{n} - \frac{1}{2}\right), \\
		q_{i,1} & = 8\left(\frac{i}{n} - \frac{1}{2}\right), \quad q_{i,2} = -8\left(\frac{i}{n} - 2\right),
	\end{align}
\end{subequations}
for $i,l = 1,\ldots,n$. For the parameters \eqref{eq:exp}, the solution to 
the continuum kernel equations \eqref{eq:infmkLk}, \eqref{eq:infmkLkbc}--\eqref{eq:infmkbcc}, 
where we choose $l_{2,1}^{(1)} = \psi_{2,1} = 0$, is explicitly given by
\begin{subequations}
	\label{eq:exk}
	\begin{align}
	K_1^1(x,\xi,y) & = y(y-1), \\
	K_1^2(x,\xi,y) & = e^{x-2\xi}y(y-1), \\
	K_2^2(x,\xi,y) & = e^{2(x-\xi)}y(y-1), \\
	L_{1,1}^1(x,\xi) & = L_{1,1}^2(x,\xi) = 0, \\
	L_{1,2}^1(x,\xi) & = 0, \quad L_{1,2}^2(x,\xi) = -2e^{x-2\xi}, \\
	L_{2,1}^2(x,\xi) & = 0, \quad	L_{2,2}^2(x,\xi) = -2e^{2(x-\xi)},
	\end{align}
\end{subequations}
where $K_1^\star(\cdot,y), L_{1,1}^\star$, and $L_{1,2}^\star$ are defined on $\mathcal{T}_1^1
= \{(x,\xi) \in [0,1]^2: \frac{1}{2}x \leq \xi \leq x \}$ and $\mathcal{T}_1^2
= \{(x,\xi) \in [0,1]^2: \xi \leq \frac{1}{2}x \}$ for the respective superindex $\star = 1,2$, while 
$K_2^2(\cdot,y), L^2_{2,1}$ and $L^2_{2,2}$ are defined on $\mathcal{T}_2^2 = 
\mathcal{T}$, for each $y\in [0,1]$. Note the discontinuity in $L_{1,2}$ along $\xi = \frac{1}{2}x$.

We implement the control law \eqref{eq:Un} with \eqref{eq:kerapp} for $n+m$ systems with 
parameters \eqref{eq:expnm} for $n=2,6,10$ (and the same $\mu, \Psi$ as in \eqref{eq:exp}). As 
the continuum kernels \eqref{eq:exk} are continuous in $y$, we approximate the exact $n+m$ 
kernels by sampling pointwise $K_1^1, K_1^2$, and $K_2^2$ at $y = 1/n, 2/n, \ldots, 1$ instead of 
using \eqref{eq:kerapp1} (see Footnote~\ref{fn:kerapp}).

For the simulation, the $n+m$ system \eqref{eq:nmm}, \eqref{eq:nmmbc} is approximated by 
finite differences with $256$ grid points in $x \in [0,1]$. The ODE resulting from the 
finite-difference approximation is solved using \texttt{ode45} in MATLAB. The initial conditions 
are $u_0^i(x) = q_{i,1} + q_{i,2}$, for $i=1,\ldots,n$, and $v_0^1(x) = v_0^2(x) = 1$, for all $x \in 
[0,1]$. The simulation results for $t \in [0, 5] $ are shown in Figures \ref{fig:U} and \ref{fig:uv}, 
which show the controls \eqref{eq:Un} for $n=2,6,10$ along with the exact controls for 
$n=10$ (computed using the kernels in \ref{app:nm}) and the solution 
components $u^{n}$ and $v^1$ for $n=10$, respectively. We note that the controls shown in 
Figure~\ref{fig:U} act as weighted averages of the solution components, but since $K_1^1, 
K_1^2$, and $K_2^2$ vanish at $y = 1$ and $L_{1,1} = L_{2,1} = 0$, the solution components 
$u^{n}$ and $v^1$ do not  affect the control law and are hence displayed separately in 
Figure~\ref{fig:uv} for $n=10$.

\begin{figure}[!ht]
	\includegraphics[width=\columnwidth]{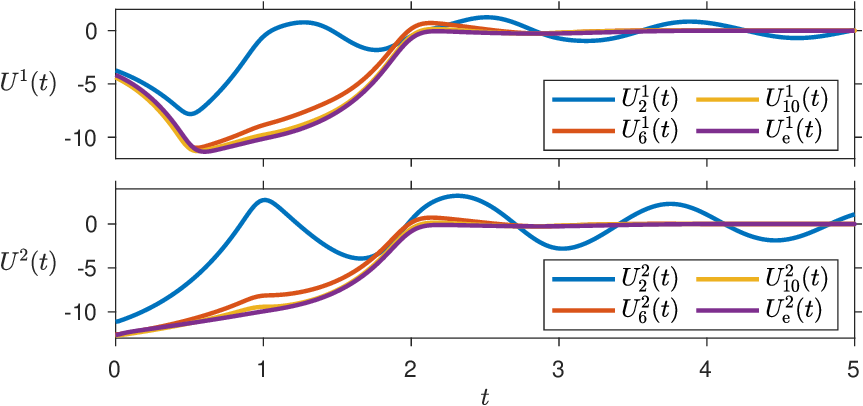}
	\caption{The controls $U(t)$ based on the approximate control law \eqref{eq:Un} for 
	$n=2,6,10$ and the respective exact control law for $n=10$.}
	\label{fig:U}
\end{figure}

\begin{figure}[!ht]
	\includegraphics[width=\columnwidth]{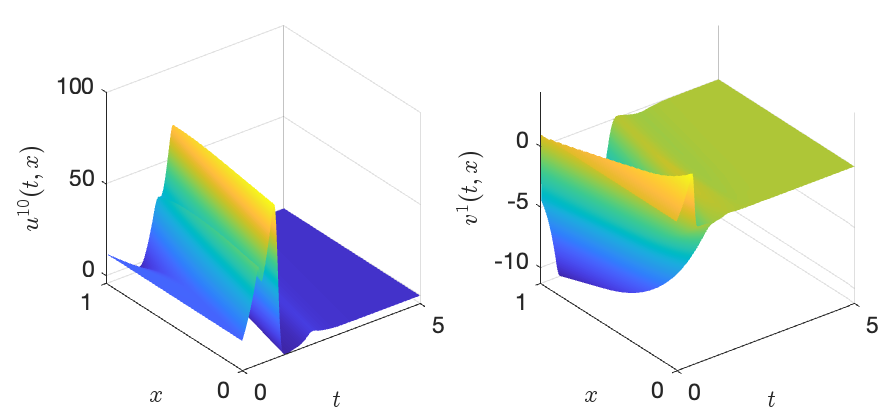}
	\caption{The solution components $u^n(t,x)$ and $v^1(t,x)$ for $n=10$.}
	\label{fig:uv}
\end{figure}

Based on Figures \ref{fig:U} and \ref{fig:uv}, we conclude that 
the control law \eqref{eq:Un} based on the continuum kernels \eqref{eq:exk} exponentially 
stabilizes the $n+m$ system for $n=2,6,10$ (and $m=2$)\footnote{Based on numerical 
simulations, the $n+2$ system with parameters \eqref{eq:expnm} is unstable for any $n \geq 
2$.}, with improved performance for larger $n$. This verifies the theoretical results. Furthermore, 
based on Figure~\ref{fig:U}, the continuum 
kernels-based control law tends close to the exact control law computed based on the $n+m$ 
kernels, as $n$ increases. We note that the $n+m$ kernels are computed based on a 
finite-difference approximation of the $n+m$ kernel equations in \ref{app:nm}, since we were not 
able to find the solution in closed form.

\section{Conclusions and Discussion} \label{sec:conc}

We introduced a backstepping control design methodology for a class of continua of hyperbolic 
PDE systems. Well-posedness of the derived kernel equations was established, together with 
exponential stability of the closed-loop system. We then utilize the continuum backstepping 
kernels for stabilization of a large-scale system counterpart, establishing that, as $n\to\infty$, the 
continuum kernels can approximate (to arbitrary accuracy) the exact backstepping kernels 
(constructed via applying backstepping to the large-scale system). This 
allowed us to prove that the control design constructed on the basis of the continuum PDE system 
can stabilize the respective large-scale system, which may be particularly useful, as, with this 
approach, complexity of computation of stabilizing kernels may not grow with the number $n$ of 
PDE systems components. This was also demonstrated in a numerical example for which the 
continuum kernels were obtained in closed form, but for which the respective, large-scale kernels 
did not exhibit a closed-form solution. We also provided a formal convergence result of the 
solutions of the large-scale PDE system to the solutions of the respective continuum.

The case $m\to\infty$ requires a quite different treatment through development of new analysis 
tools that cannot be obtained in an obvious manner via extending the tools developed here. Some 
of the main reasons for this are the following. As $m\to\infty$ the input space changes from 
$\mathbb{R}^m$ to, e.g., $L^2([0,1];\mathbb{R})$, which may impose important changes in the 
present analysis and results. In particular, as the exact, control inputs themselves (rather than the 
respective control kernels) would have to be approximated (in a certain sense), it is neither clear 
what are the stability properties of the closed-loop system one would obtain nor how to translate 
the analysis performed for finite $m$ to the case $m\to\infty$. In fact, in contrast to the present 
paper that deals with approximation of the control kernels, which gives rise to a bounded, 
vanishing perturbation that preserves exponential stability, in the case of control inputs 
approximation one may have to prove a type of practical stability with a residual value that tends 
to zero as $m\to\infty$, which would require introduction of a different, stability proof strategy (in 
particular, a respective result to Theorem~\ref{thm:solappr} of solutions' convergence may be 
essential). 
Furthermore, the characteristic curves would become 3-D regions, which makes the respective 
well-posedness analysis of the kernels much more involved. In particular, it is not obvious how 
one would then have to split the 4-D domain of evolution of the kernels into subdomains in which 
the kernels are continuous, which involves deriving 3-D discontinuity regions. In 
addition, the case $m\to\infty$ may impose important challenges in the purely technical steps. In 
particular, as the transport speeds $\mu_j$ would become a function of two variables, namely 
$\mu(x,\eta)$, it is not obvious how the assumptions made here (e.g., \eqref{eq:muass}) would 
have to be 
translated. In turn, this imposes challenges on how the well-posedness analysis of the kernels 
would have to be carried out, as, for example, one may have to properly translate the boundary 
conditions (e.g., \eqref{eq:infmkLkbca}), as well as to re-derive from scratch certain bounds (e.g., 
corresponding to 
\eqref{eq:MB1}, \eqref{eq:eps}), specifically for the case $m\to\infty$.
 
\appendix

\section{Derivation of Continuum Kernel Equations \eqref{eq:infmk}} \label{app:infmk}

\setcounter{theorem}{0}
\renewcommand{\thetheorem}{A.\arabic{theorem}}

Let us first differentiate \eqref{eq:infmV2} with respect to $x$ and use the Leibniz
rule to get
\begin{align}
\pmb{\beta}_x(t,x) & = \mathbf{v}_x(t,x) -
\mathbf{L}(x,x)\mathbf{v}(t,x)
\nonumber \\
& \qquad - \int\limits_0^1\mathbf{K}(x,x,y)u(t,x,y)dy \nonumber \\
& \qquad  -
\int\limits_0^x\mathbf{L}_x(x,\xi)\mathbf{v}(t,\xi)d\xi
\nonumber \\
& \qquad  - \int\limits_0^x \int\limits_0^1\mathbf{K}_x(x,\xi,y)u(t,\xi,y)dyd\xi.
\end{align}
Moreover, differentiating \eqref{eq:infmV2} with respect to $t$
and using \eqref{eq:infm2} gives
\begin{align}
\pmb{\beta}_t(t,x)
& = \mathbf{M}(x)\mathbf{v}_x(t,x) +
\int\limits_0^1\pmb{\Theta}(x,y)u(t,x,y)dy 
\nonumber \\
& \qquad+ \pmb{\Psi}(x)\mathbf{v}(t,x)  - 
\int\limits_0^x\mathbf{L}(x,\xi)\mathbf{M}(\xi)\mathbf{v}_\xi(t,\xi)d\xi
\nonumber \\
& \qquad - \int\limits_0^x
\mathbf{L}(x,\xi)\int\limits_0^1\pmb{\Theta}(\xi,y)u(t,\xi,y)dyd\xi
\nonumber \\
& \qquad
-\int\limits_0^x\mathbf{L}(x,\xi)\pmb{\Psi}(\xi)\mathbf{v}(t,\xi)d\xi
\nonumber \\
& \qquad + \int\limits_0^x
\int\limits_0^1\mathbf{K}(x,\xi,y)\lambda(\xi,y)u_{\xi}(t,\xi,y)dy
d\xi \nonumber \\
& \qquad - \resizebox{.71\columnwidth}{!}{$\displaystyle 
\int\limits_0^x\int\limits_0^1\mathbf{K}(x,\xi,y)
\int\limits_0^1\sigma(\xi,y,\eta)u(t,\xi,\eta)d\eta dy d\xi$}
\nonumber \\
& \qquad -
\int\limits_0^x\int\limits_0^1\mathbf{K}(x,\xi,y)\mathbf{W}(\xi,y)\mathbf{v}(t,\xi)dy
d\xi,
\end{align}
where integration by parts further gives 
\begin{align}
\int\limits_0^x\mathbf{L}(x,\xi)\mathbf{M}(\xi)\mathbf{v}_\xi(t,\xi)d\xi
& = \nonumber \\
\mathbf{L}(x,x)\mathbf{M}(x)\mathbf{v}(t,x)
- \mathbf{L}(x,0)\mathbf{M}(0)\mathbf{v}(t,0) & \nonumber \\
- \int\limits_0^x \left(\mathbf{L}_{\xi}(x,\xi)\mathbf{M}(\xi) +
\mathbf{L}(x,\xi)\mathbf{M}'(\xi)\right)\mathbf{v}(t,\xi)d\xi,
\end{align}
and
\begin{align}
\int\limits_0^x\mathbf{K}(x,\xi,y)\lambda(\xi,y)u_{\xi}(t,\xi,y)
d\xi
& = \nonumber \\
\mathbf{K}(x,x,y)\lambda(x,y)u(t,x,y) -
\mathbf{K}(x,0,y)\lambda(0,y)u(t,0,y)
& \nonumber \\
- \int\limits_0^x(\mathbf{K}_{\xi}(x,\xi,y)\lambda(\xi,y) +
\mathbf{K}(x,\xi,y)\lambda_{\xi}(\xi,y))u(t,\xi,y)d\xi.
\end{align}
Thus, in order for \eqref{eq:infmts} to hold, the kernels
$\mathbf{L}$ and $\mathbf{K}$ need to satisfy 
 \eqref{eq:infmk}, \eqref{eq:infmkbc}, 
where we also used \eqref{eq:infmtsbc1} with $\alpha(t,0,\cdot) = u(t,0,\cdot)$ and 
$\pmb{\beta}(t,0) = \mathbf{v}(t,0)$ for all $t \geq 0$. Moreover, inserting \eqref{eq:infmV} into 
\eqref{eq:infmts1} gives that $\mathbf{C}^-$ and $C^+$ need to satisfy
\begin{subequations}
\label{eq:C+-app}%
\begin{align}
  \mathbf{C}^-(x,\xi,y)
  & = \mathbf{W}(x,y)\mathbf{L}(x,\xi) +
    \int\limits_{\xi}^x
    \mathbf{C}^-(x,\zeta,y)\mathbf{L}(\zeta,\xi)d\zeta, \label{eq:C-app}\\
  C^+(x,\xi,y,\eta)
  & = \mathbf{W}(x,y)\mathbf{K}(x,\xi,\eta)
    \nonumber \\
  & \qquad + \int\limits_{\xi}^x\mathbf{C}^-(x,\zeta,
    y)\mathbf{K}(\zeta,\xi,\eta)d\zeta, \label{eq:C+app}
\end{align}
\end{subequations}
for almost all $0 \leq \xi \leq x \leq 1$ and $y,\eta \in
[0,1]$ (when applicable). Once $\mathbf{L}$ and $\mathbf{K}$ are solved from the kernel
equations \eqref{eq:infmk}, \eqref{eq:infmkbc}, then \eqref{eq:C-app}
is a Volterra equation of second kind, and well-studied in the
literature. We show in Lemma~\ref{lem:C-} that \eqref{eq:C-app}
has a well-posed solution $\mathbf{C}^- \in L^\infty(\mathcal{T}; L^2([0,1]; \mathbb{R}^{1\times 
m}))$. Once $\mathbf{C}^-$ is solved from \eqref{eq:C-app}, $C^+$ is 
explicitly given as a function of $\mathbf{W}, \mathbf{K}$ and
$\mathbf{C}^-$ by \eqref{eq:C+app}, by which $C^+ \in 
L^\infty(\mathcal{T};L^2([0,1]^2;\mathbb{R}))$ follows.

\begin{lemma}
  \label{lem:C-}
  Under Assumption~\ref{ass:infm}, the equation \eqref{eq:C-app} admits a unique solution 
  $\mathbf{C}^-  \in L^\infty(\mathcal{T}; L^2([0,1]; \mathbb{R}^{1\times 
  	m}))$.
\begin{proof}
Utilizing similar tools as in \cite[Thm 2.3.5]{HocBook} and
\cite[Lem. A.2]{HuLVaz19}, we show that $\mathbf{C}^-$ is given
by the series 
\begin{equation}
  \label{eq:C-ser}
  \mathbf{C}^-(x,\xi,y) = \sum_{k=0}^{\infty} \Delta \mathbf{C}_k^-(x,\xi,y),
\end{equation}
where $\Delta \mathbf{C}_0^-(x,\xi,y) =
\mathbf{W}(x,y)\mathbf{L}(x,\xi)$, so that $\Delta\mathbf{C}_0 \in C(\mathcal{T}_i^p; L^2([0,1]; 
\mathbb{R}))$ for any $1 \leq i \leq p \leq m$
by Theorem~\ref{thm:infmkwp}, and $\Delta \mathbf{C}_k^-$ for $k
\geq 1$ is defined recursively by 
\begin{equation}
  \label{eq:C-iter}
  \Delta \mathbf{C}_k^-(x,\xi,y) = \int\limits_{\xi}^x \Delta
  \mathbf{C}_{k-1}^-(x,\zeta,y) \mathbf{L}(\zeta,\xi)d\zeta,
\end{equation}
by which $\Delta \mathbf{C}_k^- \in C(\mathcal{T}_i^p;
L^2([0,1];\mathbb{R}))$ for $k \geq 1$. By induction, it immediately
follows that $\Delta \mathbf{C}_k$ satisfy
\begin{align}
  \label{eq:C-ind}
  \max_{j\in \left\{ 1,\ldots,m \right\}}\esssup_{(x,\xi) \in
  \mathcal{T}}\|\left( \Delta C_j^-\right)_k(x,\xi,\cdot)\|_{L^2}
  & \leq \nonumber \\
   M_W\frac{(M_L)^{k+1}(x-\xi)^k}{k!}, &
\end{align}
where 
\begin{equation}
  \label{eq:ML}
  M_L = \max_{i,j \in \left\{ 1,\ldots,m \right\}} \esssup_{(x,\xi)
    \in \mathcal{T}} \left| L_{i,j}(x,\xi) \right|,
\end{equation}
and $M_W$ is given in \eqref{eq:MW}. Thus, the series 
\eqref{eq:C-ser} converges on $L^\infty(\mathcal{T}; L^2([0,1];\mathbb{R}))$ to the stated  
solution to \eqref{eq:C-app}. 
\end{proof}
\end{lemma}

\section{Invertibility of \eqref{eq:infmV}} \label{app:invk}

\begin{lemma}
\label{lem:invk}
Under Assumption~\ref{ass:infm}, the transformation \eqref{eq:infmV} is boundedly invertible on 
$E_c$. 
\begin{proof}
The claim follows after solving for $(u, \mathbf{v})$ from \eqref{eq:infmV}. Since 
 $u = \alpha$, inserting this to \eqref{eq:infmV2} gives
\begin{align}
\label{eq:vab}
  \mathbf{v}(t,x) -
\int\limits_0^x\mathbf{L}(x,\xi)\mathbf{v}(t,\xi)d\xi & = \nonumber
\\
\int\limits_0^x \int\limits_0^1
\mathbf{K}(x,\xi,y)\alpha(t,\xi,y)dyd\xi - \pmb{\beta}(t,x), &
\end{align}	
which is a Volterra equation of second kind for $\mathbf{v}(t,\cdot)$ in terms of $\alpha(t,\cdot), 
\pmb{\beta}(t,\cdot), \mathbf{L}$, and $\mathbf{K}$, for any (fixed) $t \geq 0$. Since $(\alpha, 
\pmb{\beta}) \in C([0, +\infty); E_c)$, being the solution to \eqref{eq:infmts}, \eqref{eq:infmtsbc} 
(by Theorem~\ref{thm:stab}),
and as $\mathbf{K} \in L^\infty(\mathcal{T}; L^2([0,1];\mathbb{R}^m)$, $\mathbf{L} \in 
L^\infty(\mathcal{T}; \mathbb{R}^{m\times m})$ by Theorem~\ref{thm:infmkwp}, the equation 
\eqref{eq:vab} has a unique solution $\mathbf{v}(t,\cdot) \in L^2([0,1]; \mathbb{R}^m)$ for all $t 
\geq 0$ by \cite[Thm 2.3.6]{HocBook}.
\end{proof}
\end{lemma}

\section{Kernel Equations for Linear Hyperbolic $n+m$ PDEs} \label{app:nm}

Denote $\mathbf{k}_i^p = \left(k_{i,l}^p\right)_{l=1}^n$ and $\pmb{\ell}_i^p = 
\left(\ell_{i,j}^p\right)_{j=1}^m$, where $k_{i,l}^p, \ell_{i,j}^p$ for $l = 1,\ldots,n, j =1,\dots,m$, and 
$1 \leq i \leq p \leq m$ denote the $n+m$ kernels restricted to $\mathcal{T}_i^p$. Using the 
notation of \eqref{eq:nmm}, \eqref{eq:nmmbc}, these satisfy 
the kernel equations (cf. \cite[(A.19)--(A.23)]{HuLVaz19})
\begin{subequations}
\label{eq:nmk}
\begin{align}
	\mu_i(x) \partial_x \mathbf{k}_i^p(x,\xi) - \pmb{\Lambda}(\xi)\partial_\xi \mathbf{k}_i^p(x,\xi) - 
	\pmb{\Lambda}'(\xi) \mathbf{k}_i^p(x,\xi) & = \nonumber \\
	\frac{1}{n}\mathbf{\Sigma}^T(\xi)\mathbf{k}_i^p(x,\xi) + \pmb{\Theta}^T(\xi)\pmb{\ell}_i^p(x,\xi),
	& \label{eq:nmkk} \\
	\mu_i(x)\partial_x \pmb{\ell}_i^p(x,\xi) + \mathbf{M}(\xi)\partial_\xi\pmb{\ell}_i^p(x,\xi) + 
	\mathbf{M}'(\xi)\pmb{\ell}_i^p(x,\xi) & = \nonumber \\
	\frac{1}{n}\mathbf{W}^T(\xi)\mathbf{k}_i^p(x,\xi) + \pmb{\Psi}^T(\xi)\pmb{\ell}_i^p(x,\xi),
	& \	\label{eq:nmkl} 
\end{align}
\end{subequations}
on $\mathcal{T}_i^p$, $1 \leq i \leq p \leq m$, with boundary conditions
\begin{subequations}
	\label{eq:nmkbc}%
	\begin{align}
		\mu_i(x)\mathbf{k}_i^i(x,x) + \pmb{\Lambda}(x)\mathbf{k}_i^i(x,x) & = 
		-\pmb{\Theta}_{i,\cdot}^T(x),  
		\label{eq:nmkbc1} \\
		\mu_i(x)\pmb{\ell}_i^i(x,x) - \mathbf{M}(x)\pmb{\ell}_i^i(x,x) & = -\pmb{\Psi}_{i,\cdot}^T(x), 
		\label{eq:nmkbc2} 
		\\
		\frac{1}{n}\mathbf{Q}^T\pmb{\Lambda}(0)\mathbf{k}_i^m(x,0) - 	
		\mathbf{M}(0)\pmb{\ell}_i^m(x,0) & = \mathbf{g}_i(x),  \label{eq:nmkbc3}
	\end{align}
\end{subequations}
where $\mathbf{g}_i = \left(g_{i,j}\right)_{j=1}^m$ satisfies $g_{i,j} = 0$ for $i \leq j$.
Additionally, an artificial boundary condition is imposed to guarantee
well-posedness of the kernel equations as follows
\begin{equation}
	\label{eq:nmabc}
	\forall j < i: \quad \ell_{i,j}^p(1,\xi) = l_{i,j}(\xi),
\end{equation}
where the functions $l_{i,j}$ can be chosen arbitrarily. However, we
choose $l_{i,j}$ such that
\begin{equation}
	\label{eq:ltpm}
	l_{i,j}(1) = -\frac{\psi_{i,j}(1)}{\mu_i(1)- \mu_j(1)},
\end{equation}
in order for \eqref{eq:nmabc} to coincide with \eqref{eq:nmkbc2} at $x=1$ (see \cite[Rem. 
6]{HuLDiM16}). Finally, the segmented kernels are subject to continuity conditions 
\begin{subequations}
	\label{eq:nmkbcc}
	\begin{align}
		\forall i < p,\forall j\neq p: & &  \ell_{i,j}^{p-1}(x,\rho_i^p(x))
		& = \ell_{i,j}^p(x,\rho_i^p(x)), \\
		\forall i < p: & & \mathbf{k}_{i}^{p-1}(x,\rho_i^p(x))
		& = \mathbf{k}_{i}^p(x,\rho_i^p(x)),
	\end{align} 
\end{subequations}
for $1 \leq i \leq p \leq m$ and $j = 1,\ldots, m$. It follows by \cite[Thm 
A.1]{HuLVaz19} that the kernel equations \eqref{eq:nmk}--\eqref{eq:nmkbcc} have well-posed 
solutions, which are additionally continuous on every $\mathcal{T}_i^p$.


\end{document}